\documentclass[11pt, reqno, a4paper, final]{amsart}

\usepackage[applemac]{inputenc}
\usepackage[english]{babel}
\usepackage[T1]{fontenc}
\usepackage{amsmath,amsfonts,amssymb,amsthm,amscd}
\usepackage{eucal}
\usepackage{array}
\usepackage{color}
\usepackage{hyperref}
\usepackage{a4wide}										 	
\usepackage{mathrsfs}									  
\usepackage[all]{xy}										
\usepackage{url}												
\usepackage{verbatim,epsfig}
\usepackage{xfrac} 

\newcounter{claim-counter}

\theoremstyle{plain}

\newtheorem{theorem}{Theorem}[section] 
\newtheorem{cor}[theorem]{Corollary}
\newtheorem{corollary}[theorem]{Corollary}
\newtheorem{lem}[theorem]{Lemma}
\newtheorem{lemma}[theorem]{Lemma}
\newtheorem{prop-defi}[thm]{Definition \& Proposition}
\newtheorem{prop}[theorem]{Proposition}

\newtheorem*{thm*}{Theorem}
\newtheorem*{prop*}{Proposition}
\newtheorem*{cor*}{Corollary}
\newtheorem{proposition}[theorem]{Proposition}

\theoremstyle{definition}
\newtheorem{defi}[theorem]{Definition}
\newtheorem{definition}[theorem]{Definition}

\newtheorem{example}[theorem]{Example}
\newtheorem{notation}[theorem]{Notation}

\newtheorem{rem}[theorem]{Remark}
\newtheorem{remark}[theorem]{Remark}

\newtheorem{claim}[claim-counter]{Claim}
\newtheorem*{claim*}{Claim}

\newcommand{\NN}{{\mathbb N}}
\newcommand{\ZZ}{{\mathbb Z}}
\newcommand{\BB}{{\mathbb B}}

\newcommand{\RR}{{\mathbb R}}
\newcommand{\CC}{{\mathbb C}}

\newcommand{\F}{{\mathcal F}}
\newcommand{\E}{{\mathcal E}}

\newcommand{\A}{{\mathcal A}}

\renewcommand{\H}{\mathcal{H}}

\renewcommand{\L}{{\mathcal L}}

\newcommand{\K}{{\mathcal K}}

\renewcommand{\max}{{\operatorname{max}}}

\newcommand{\varps}{{\varepsilon}}

\newcommand{\tens}{\otimes}

\newcommand{\spann}{{\operatorname{span}}}

\newcommand{\To}{\longrightarrow}

\newcommand{\Hom}{\operatorname{Hom}}

\newcommand{\id}{\operatorname{id}}

\newcommand{\ev}{\operatorname{ev}}

\newcommand{\del}{{\partial}}

\newcommand{\Pol}{{\operatorname{Pol}}}

\renewcommand{\min}{{\operatorname{min}}}

\renewcommand{\leq}{\leqslant}
\renewcommand{\geq}{\geqslant}

\newcommand{\bbb}{{\mathbf{1}}}

 \newcommand{\mfg}{{\mathfrak{g}}}

\newcommand{\Cohom}{{\operatorname{H}}}

\newcommand{\Cohomred}{\underline{\Cohom}}

\renewcommand{\restriction}{\rvert}

\usepackage{graphicx}
\usepackage{amsfonts}
\usepackage[backrefs]{amsrefs}
\usepackage{enumerate}
\usepackage{mathrsfs}
\usepackage{color}

\numberwithin{equation}{section}

\usepackage{setspace,nicefrac}

\newcommand{\PolCohom}[5]{\operatorname{H}_{(#1) #3}^{#2}(#4,#5)}

\newcommand{\ldiff}{\operatorname{\partial}}
\newcommand{\rdiff}{\operatorname{\scalebox{-1}[1]{$\partial$}}}

\newcommand{\minusdot}{\dot{-}}
\newcommand{\plusdot}{\dot{+}}

\newcommand{\multiindex}[1]{\mathbf{#1}}
\newcommand{\Multiindex}[3]{\multiindex{#1}_{#2,#3}}
\newcommand{\Bultiindex}[2]{\multiindex{#1}_{#2}}
\newcommand{\multiclass}{\operatorname{\mathbf{cl}}}
\newcommand{\multirank}{\operatorname{\mathbf{rk}}}
\newcommand{\multidim}{\operatorname{\mathbf{dim}}}
\newcommand{\inv}{{\operatorname{inv}}}

\usepackage{thmtools}
\declaretheorem[style=theorem,name={Theorem}]{theoremletter}

\newtheorem{introcorollary}[theoremletter]{Corollary}

\title{Polynomial cohomology and polynomial maps on nilpotent groups}

\author{David Kyed}
\address{David Kyed, Department of Mathematics and Computer Science, University of Southern Denmark, Campusvej 55, DK-5230 Odense M, Denmark}
\email{dkyed@imada.sdu.dk}

\author{Henrik Densing Petersen}

\subjclass[2010]{57M07, 22E41, 22E25} 
\keywords{Locally compact groups, nilpotent groups, group  cohomology, polynomials}

\begin{document}

\begin{abstract}
We  introduce a refined version of group cohomology  and relate it  to the space of polynomials on the group in question. We show that the polynomial cohomology with trivial coefficients admits a description in terms of ordinary cohomology with polynomial coefficients, and that the degree one polynomial cohomology with trivial coefficients admits a description directly in terms of polynomials.  Lastly, we give a complete description of the polynomials on a connected, simply connected nilpotent Lie group  by showing that these are exactly the maps that pull back to classical polynomials via the exponential map.

\end{abstract}

\maketitle

\section{Introduction}
Group cohomology is by now a standard tool with a wide range of applications spanning from finite to locally compact groups and across a variety of disciplines, including differential geometry,  ergodic theory, topology and operator algebras. 
The aim of the present paper is to introduce a refined version of group cohomology, dubbed polynomial cohomology, which consists of a family of functors $\Cohom_{(d)}^n(G, -)$ for which the case $d=1$ corresponds to the ordinary cohomology $\Cohom^n(G, -)$ of the group $G$ in question. As the name suggests, this cohomology theory is intimately linked with polynomials  on  groups (see Section \ref{sec:polynomial_maps_on_groups} for definitions), a notion that dates back to (at least)  the works of Passi from the 1960's \cite{Passi-I, Passi-II, Passi-functor}. Passi's work  is primarily concerned with polynomials  on discrete abelian groups, and already in his works the connection to cohomology theory appears, in that he obtains information about  circle-valued 2-cohomology of abelian groups as a consequence of his results \cite[Theorem 4.1]{Passi-II}. Polynomial maps also appear in the work of Buckley \cite{Buckley-1} regarding nilpotency of wreath products, and later in the work of Leibman \cite{Leibman:Polynomial_mappings} where emphasis is also on the case of nilpotent groups.
The setting of the present paper will be that of  locally compact, second countable groups
and in the case of trivial coefficients the relationship between polynomial cohomology and actual polynomials is made precise by means of the following theorem.

\begin{theoremletter}[see Proposition \ref{prop:PolCohom_linear_description} and \ref{prop:HdP1_R_description}]\label{thm:A}
Let $G$ be a locally compact, second countable group and  denote by $\Pol_d(G)$ the space of polynomials on $G$ of degree at most $d$.  
Then   there exists an isomorphism of topological vector spaces
\[
\Cohom_{(d)}^n\left(G,\RR\right) \simeq \Cohom^n\left(G, \Pol_{d-1}(G)\right),  \ \text{ for all } n\in \NN_0.
\]
Moreover, in degree 1 there exists an isomorphism of topological vector spaces 
\[
\Cohom_{(d)}^1\left(G,\RR\right)   \simeq \Pol_d(G)/\Pol_{d-1}(G).
\]
\end{theoremletter}
In the case where $G$ is a discrete abelian group,  $ \Pol_d(G)/\Pol_{d-1}(G)$ was actually studied from a functorial point of view already by Passi \cite{Passi-functor} although it was not  considered as a cohomology theory in the sense of the present paper. As a consequence of Theorem \ref{thm:A} we deduce the following:

\begin{introcorollary}[See Corollary \ref{cor:cohom_fin_dim_implies_pol_fin_dim}]
If $G$ is a cohomologically finite dimensional, locally compact, second countable group then $\Pol_d(G)$ is finite dimensional for each $d\in \NN_0$.
\end{introcorollary}
Here $G$ is called  \emph{cohomologically finite dimensional} if its cohomology groups with coefficients in finite dimensional $G$-vector spaces are finite dimensional themselves, and we remark that this for instance includes  finitely generated discrete groups whose classifying space is a finite CW complex as well as connected simply connected nilpotent Lie groups.\\

Often the polynomial cohomology captures no new information about the group (this is for instance the case if the group has compact abelianization; cf.~Remark \ref{rem:compact-abelianization}), but for nilpotent groups we show that  the situation is quite different:

\begin{theoremletter}[see Theorem \ref{thm:pol_Malcev_description} \& Remark \ref{rem:exp-description}]\label{thm:B} 
For a connected, simply connected, nilpotent Lie group $G$ the  polynomials on $G$ are exactly the functions that pull back to ordinary polynomials on the Lie algebra of $G$ via the exponential map.
\end{theoremletter}

In the setting of Theorem \ref{thm:B},   by considering $G$ as the set of real points of a linear algebraic group, this shows, in particular, that the space of polynomials in this case is nothing but the classical set of regular functions on the algebraic group in question (see also Remark \ref{rm:alg-geom-rem}).
Using Theorem  \ref{thm:B}, we also  deduce  that the space of all polynomials on a connected simply connected Lie group $G$ is a Hopf algebra and that this is a complete invariant of the Lie group in question:

\begin{introcorollary}[See Theorem \ref{thm:hom_existence}]\label{cor:intro-cor}
Let $G$ and $H$ be 
connected, simply connected, nilpotent Lie groups
 and suppose that $\Psi \colon \operatorname{Pol}(G) \rightarrow \operatorname{Pol}(H)$ is a Hopf algebra homomorphism. Then there is a unique continuous group homomorphism $\varphi\colon H \rightarrow G$ such that $\Psi$ is induced by $\varphi$, and $\varphi$ is an isomorphism if and only if $\Psi$ is. 
\end{introcorollary}
We note that the result in Corollary \ref{cor:intro-cor} is not new, in the sense that Theorem \ref{thm:B}  allows us to think of $\Pol(G)$ as the algebra of regular functions on $G$ when the latter is considered as the set of  real points of an algebraic group, and after adapting this point of view,   the result in Corollary \ref{cor:intro-cor} is then 
 a classical fact in algebraic geometry (cf.~\cite[Chapter 1]{Borel-GTM}). Note, however, that the proof of  Corollary \ref{cor:intro-cor} provided in Section \ref{sec:algebra_pol} makes no (explicit) use of algebraic geometry, and it is our hope that this will make the result accessible to a different audience.


\subsection*{Acknowledgements}
We are grateful to Nicolas Monod and Ryszard Nest for several interesting discussions, and for their hospitality and generosity. We would also like to thank Yves Cornulier and Peter Schlicht for a number of useful comments. The first named author gratefully acknowledges the financial support from the Lundbeck foundation (grant R69-A7717), the Villum foundation (grant 7423) and DFF (grant 7014-00145B). Lastly, we would like to thank the anonymous referee for a very thorough reading of the paper, for pointing out a number of relevant references and for numerous comments and corrections that significantly improved the exposition of the paper.

\section{Notation and conventions}\label{sec:notation}

\subsubsection*{Topological vector spaces} 
Unless explicitly stated otherwise, all generic topological vector spaces are implicitly assumed to be Hausdorff. Our primary need for treating non-Hausdorff topological vector spaces stems from the fact that the cohomology of a topological group is (generally) a non-Hausdorff topological vector space, but this will not lead to any confusion, as it will always be clear from the context whether or not the space in question is (assumed) Hausdorff.
A \emph{morphism} $\varphi \colon \mathcal{E} \rightarrow \mathcal{F}$ between (not necessarily Hausdorff) topological vector spaces $\mathcal{E}$ and $\mathcal{F}$ is a continuous linear map, and an \emph{isomorphism} is a morphism such that there exists an inverse morphism $\varphi^{-1}\colon \mathcal{F}\rightarrow \mathcal{E}$.

\subsubsection*{Topological groups} The term `group'  will always mean an abstract group without any topology, and we will follow the standard convention and abbreviate `locally compact second countable  by `lcsc'.
 We denote the identity element in a group $G$ by $\mathbf{1}_G$, leaving out the subscript whenever this does not lead to ambiguity. Lastly, we will denote the inversion map $g\mapsto g^{-1}$ by $\inv\colon G\to G$ whenever notationally convenient, and  the center of $G$ by $Z(G)$. The product map is occasionally denoted by $m\colon G\times G \to G$ and the map $(g,h) \mapsto gh^{-1}$ by $\tilde{m}$.

\subsubsection*{Topological $G$-modules} Let $G$ be a lcsc group. By a \emph{topological (or continuous) $G$-module} we shall mean a (Hausdorff) topological vector space $\E$ over either $\RR$ or $\CC$ together with an action of $G$ by invertible linear maps such that the action map $G\times \E \to \E $ is continuous. Note that when $\E$ is a Hilbert space this, a priori quite strong continuity requirement,  coincides with the more familiar notion of a strongly continuous $G$-action \cite[Lemme D8]{Guichardet:Cohomologie}.
A \emph{morphism} of topological $G$-modules is a morphism of the underlying topological vector spaces which intertwines their respective $G$-actions.

\subsubsection*{Extended natural numbers} We denote by $\NN_0$ the set $\NN\cup \{0\}$, and, following Leibman \cite[3.3]{Leibman:Polynomial_mappings}, we denote  the set $\{ -\infty \} \cup \mathbb{N}_0$ by $\mathbb{Z}_*$ and define
\begin{align*}
x\plusdot y &:= \left\{ \begin{array}{cc} x+y, & x,y\in \mathbb{N}_0 \subseteq \mathbb{Z}_* \\ -\infty, & \textrm{if either of } \, x,y = -\infty \end{array} \right. \\
x\minusdot y& := \left\{ \begin{array}{cc} x-y, & x\geq y\in \mathbb{N}_0 \subseteq \mathbb{Z}_* \\ -\infty, & x = -\infty \textrm{ or } x < y \in \mathbb{N}_0\end{array} \right.
\end{align*}
We leave $x\minusdot y$ undefined if $y=-\infty$.

\subsubsection*{Multi-index notation}
Let  $I=\{i_1,\dots, i_l\}$  be a finite set endowed with a fixed total order such that $i_1<i_2< \dots < i_l$.  By a \emph{multi-index} over a $I$
 we mean  an element 
 \[
 \multiindex{d} = (d_i)_{i\in I}=(d_{i_1},\dots, d_{i_l}) \in \mathbb{N}_0^I. 
 \]
 For a Mal'cev group $G$ (see Section \ref{sec:Malcev} for the definition of Mal'cev groups; in particular this includes connected, simply connected, nilpotent Lie groups) we denote by $\multiclass(G)$ the multi-index $(1,\dots ,\operatorname{cl}(G))\in \mathbb{N}_0^{\operatorname{cl}(G)}$ and by $\multirank(G)$ the multi-index $(\dim_{\mathbb{R}} \mathfrak{g}_{[i]}/\mathfrak{g}_{[i+1]})_{i=1,\dots ,\operatorname{cl}(G)} \in \mathbb{N}_0^{\operatorname{cl}(G)}$. For a multi-index $\multiindex{k}$ over $I$ we denote by $\mathbb{N}_0^{\multiindex{k}}$ the product set $\prod_{i\in I} \mathbb{N}_0^{k_i}= \mathbb{N}_0^{k_{i_1}}\times \dots \times \mathbb{N}_0^{k_{i_l}}$, and for a Mal'cev group $G$ we denote by $\multidim(G)$ the multi-index (where we write $\multiindex{m}:=\multirank(G)$)
\begin{equation}
\multidim(G) := ( (1)_{j=1,\dots ,m_1},\dots ,(\operatorname{cl}(G))_{j=1,\dots ,m_{\operatorname{cl}(G)}} ) \in \mathbb{N}^{\multiindex{m}}_0. \notag
\end{equation}
For any $d\in \mathbb{N}_0$ and any multi-index $\multiindex{k}$ over $I$ we define
\begin{equation}
\Multiindex{D}{d}{\multiindex{k}} := \Big\{ \multiindex{d}\in \mathbb{N}_0^I \mid  \sum_{i\in I} k_id_i \leq d \Big\}, \notag
\end{equation}
and denote by $\Multiindex{D}{d}{\multiindex{k}}^{=}$ the subset for which equality holds. Finally, we set
\begin{equation}
\Bultiindex{B}{\multiindex{k}} := \dot{\bigcup}_{i\in I} \{ 1,2,\dots ,k_i \}, \notag
\end{equation}
(disjoint union) and consider on this set the lexicographic order. 

\subsubsection*{Product notation}
For a group $G$, a finite (totally) ordered set $I=\{i_1, \dots, i_l\}$ with $i_1< \dots < i_l$ and a map $I \ni i \mapsto g_i \in G $ we write $\prod_{i\in I} g_i$ for the element 
$g_{i_1}g_{i_2}\cdots g_{i_l}\in G$.

\section{Polynomial cohomology of lcsc groups} \label{sec:Polynomial_cohomology}
In this section we recall the definition of continuous cohomology for locally compact groups,  and define, more generally, a notion of \emph{polynomial cohomology}, for which the ``linear'' (or degree one) case coincides with the usual cohomology.


\begin{definition}[strengthened morphism] \label{def:strengthened_morphism}
A morphism $v\colon \E\to \F$ between topological $G$-modules is said to be \emph{strengthened} if there exists a morphism of topological vector spaces $\eta\colon \F\to \E$ such that $ v \circ \eta \circ v = v$. 
\end{definition}
We emphasize that it is \emph{not} part of the definition that the map $\eta$ be $G$-equivariant. The definition of a strengthened morphism just given might not be completely standard, but is easily seen to be equivalent with the one used e.g.~in \cite[Chapter III \& Appendix D]{Guichardet:Cohomologie}; in particular, for injective morphisms, being strengthened is the same as being left invertible in the category of topological vector spaces.

\begin{definition}[relative injectivity] \label{def:relatively_injective}
A continuous $G$-module $\mathcal{E}$ is said to be \emph{relatively injective} if, given any diagram
\begin{equation}
\xymatrix{ 0 \ar[r] & \mathcal{F}_1 \ar[d]^v \ar[r]^u & \mathcal{F}_2 \ar@{-->}[dl]^{\exists  w} \\ & \mathcal{E} } \notag
\end{equation}
where $u\colon \mathcal{F}_1 \rightarrow \mathcal{F}_2$ is a strengthened injective morphism, there exists a morphism $w\colon \mathcal{F}_2 \rightarrow \mathcal{E}$ such that the augmented diagram commutes.
\end{definition}

For $G$ a lcsc group and $X$ a locally compact space on which $G$ acts continuously by homeomorphisms, the space of continuous functions $C(X,\mathcal{E})$ is a continuous $G$-module for every continuous $G$-module $\mathcal{E}$, when endowed the standard action $(g.f)(x) = g.f(g^{-1}.x)$. Recall that the topology on $C(X,\mathcal{E})$ is the projective topology generated by the restriction maps $C(X,\mathcal{E}) \rightarrow C(K,\mathcal{E})$ over all compact subsets $K$ of $X$, that is, the topology of uniform convergence on compact sets. In particular, note that if $X$ is second countable and $\mathcal{E}$ is a Fr{\'e}chet space (as will often be the case in this paper), then $C(X,\mathcal{E})$ is a Fr{\'e}chet space as well.

\begin{lemma}[{\cite[III, Proposition 1.2]{Guichardet:Cohomologie}}]
Let $G$ be a lcsc group and $\mathcal{E}$ be a continuous $G$-module. Then $C(G,\mathcal{E})$ is relatively injective. In particular, the category of continuous $G$-modules contains sufficiently many relatively injectives. Further, there is a strengthened, relatively injective resolution 
\begin{equation}
\xymatrix{ 0\ar[r] & \mathcal{E} \ar[r]^>>>>>{\varps} & C(G,\mathcal{E}) \ar[r]^>>>>{d^0} & C(G\times G,\mathcal{E}) \ar[r]^>>>>{d^1} & \cdots \ar[r]^>>>>{d^{n-1}} & C(G^{n+1},\mathcal{E}) \ar[r]^>>>>{d^n} & \cdots }.\notag
\end{equation}
where the coboundary maps are given by $\varps(\xi)(g):=\xi$ and 
\begin{align}\label{standard-resolution-coboundary-map}
d^n(f) (g_0,\dots, g_{n+1}) &:= \sum_{i=0}^{n+1} (-1)^i f(g_0,\dots, \hat{g}_i,\dots, g_{n+1}),
\end{align}
where the symbol  $\hat{g}_i$ denotes omission of the element $g_i$.
\end{lemma}


\begin{definition}[differential notation and higher order invariants]\label{def:differential_notation}
Let $G$ be a lcsc group and $\mathcal{E}$ be a continuous $G$-module. For $g\in G$, we denote by $\ldiff_g\colon \E\to \E $ the continuous, linear map $\xi \mapsto g.\xi-\xi$, and for $d\in \mathbb{N}$ we define the \emph{$d$'th order invariants} in $\E$ as
\begin{equation}
\mathcal{E}^{G(d)} := \{ \xi \in \mathcal{E} \mid \forall g_1,\dots , g_d \in G : \ldiff_{g_1} \circ \cdots \circ \ldiff_{g_d} (\xi) = 0 \} .\notag
\end{equation}
\end{definition}
Note that $\mathcal{E}^{G(d)}$ is the pre-image under the quotient map $\E\to \E/\E^{G(d-1)}$ of the subspace $\left( \mathcal{E} / \mathcal{E}^{G(d-1)} \right)^G$ --- an observation we will be using frequently (without reference) in the sequel.\\
The relation $ (g_1-\bbb)\cdots (g_d-\bbb)(gx)= g(g^{-1}g_1g-\bbb)\cdots (g^{-1}g_dg-\bbb)x $ shows that $\E^{G(d)}$ is a (closed) $G$-invariant subspace in $\E$ and hence $(-)^{G(d)}$ defines an endo-functor on the category of topological $G$-modules, which recovers the classical invariants functor when $d=1$.
Furthermore it is easy to see that $(-)^{G(d)}$ is left exact, and thus has well defined right-derived functors, and these are the object of study in this section. We spell out this construction by means of the following:

\begin{definition}[continuous polynomial cohomology]
Let $G$ be a lcsc group and let $\mathcal{E}$ be a continuous $G$-module. For  $d\in \mathbb{N}$, we define the \emph{$d$-th order continuous polynomial cohomology of $G$ with coefficients in $\mathcal{E}$} as
\begin{equation}
\PolCohom{d}{n}{}{G}{\mathcal{E}} := \frac{\ker \left(d^n \vert_{\mathcal{E}_n^{G(d)}} \right)} {\operatorname{im} \Big( d^{n-1} \vert_{\mathcal{E}_{n-1}^{G(d)}} \Big)}, \quad n\in \mathbb{N}_0, \notag
\end{equation}
where $\xymatrix{ 0 \ar[r] & \mathcal{E} \ar[r] & (\mathcal{E}_{\bullet}, d^{\bullet}) }$ is any strengthened, relatively injective resolution of $\mathcal{E}$. 
The space $\ker (d^n \vert_{\mathcal{E}_n^{G(d)}} )$ is denoted by $\operatorname{Z}^n_{(d)}(G,\E)$ and referred to as the space of \emph{homogeneous (degree $d$) polynomial $n$-cocycles} and the space
$\operatorname{im} ( d^{n-1} \vert_{\mathcal{E}_{n-1}^{G(d)}} )$ is denoted by $\operatorname{B}^n_{(d)}(G,\E)$ and referred to as the space of \emph{homogeneous (degree $d$) polynomial $n$-coboundaries}. 
\end{definition}
The left-exactness of $(-)^{G(d)}$ combined with standard arguments in relative homological algebra  (cf.~\cite[III, Corollaire 1.1]{Guichardet:Cohomologie}) implies that the polynomial cohomology $\PolCohom{d}{n}{}{G}{\mathcal{E}}$ is indeed well  defined as a (generally non-Hausdorff) topological vector space; that is, using different relatively injective, strengthened resolutions to compute $\PolCohom{d}{n}{}{G}{\mathcal{E}}$ yields bijective, bi-continuous, linear maps between the resulting polynomial cohomology spaces. \\

When $d=1$, one sees that we recover the ordinary cohomology of $G$ (see e.g.~\cite[III]{Guichardet:Cohomologie}) which we will denote by $\Cohom^n(G,\E)$, as is more standard. Moreover, we will denote by $\underline{\Cohom}^n(G,\E)$ the \emph{reduced} cohomology of $G$; i.e.~the maximal Hausdorff quotient of the topological vector space $\Cohom^n(G,\E)$. Before we continue our investigation, we record a definition that will be needed in the sections to follow.

\begin{definition} \label{def:cohom_fin_dim}
We say that a lcsc group $G$ is \emph{cohomologically finite dimensional} if $\Cohom^n(G, \mathcal{E})$ is finite dimensional for every $n \in \mathbb{N}_0$ whenever $\E$  is a finite dimensional continuous $G$-module.
  \end{definition}
Note that any finitely generated discrete group whose classifying space is a finite  CW complex is an example of a cohomologically finite dimensional group. So are connected Lie groups (e.g.~by the van Est theorem \cite[III, Corollaire 7.2]{Guichardet:Cohomologie}).

\begin{remark}\label{rem:compact-abelianization}
A direct computation shows that, given any topological $G$-module $\mathcal{E}$ and any $\xi\in \mathcal{E}^{G(2)}$, the map $g\mapsto \ldiff_g.\xi$ is a continuous homomorphism from $G$ to $\mathcal{E}$. Thus, if $G$ has compact abelianization we conclude that $\mathcal{E}^{G(2)} = \mathcal{E}^G$ for every topological $G$-module $\mathcal{E}$, and inductively that $\mathcal{E}^{G(d)} = \mathcal{E}^G$ for all  $d$. Hence, for such $G$, the continuous polynomial cohomology coincides with the ordinary continuous cohomology in the sense that $\Cohom^n_{(d)}(G,\E)= \Cohom^n(G,\E)$ for all $d\in \NN$.

\end{remark}

\subsection{Polynomial cohomology in terms of ordinary cohomology}\label{sec:taut-self-coupling}
Our next aim is the following proposition which 
 gives a description of polynomial cohomology in terms of ordinary cohomology.  In the statement we  write $\operatorname{Pol}_{d-1}(G)$ for the space $C(G,\mathbb{R})^{G(d)}$, where the $G(d)$-invariants are taken  with respect to the right-regular representation; see  Section \ref{sec:polynomial_maps_on_groups} below, for an explanation of this terminology.

\begin{proposition} \label{prop:PolCohom_linear_description}
Let $G$ be a lcsc group and let $d\in \mathbb{N}$. Then:
\begin{enumerate}[(i)]
\item There are isomorphisms $\tau^{\bullet}\colon \Cohom^\bullet (G, \operatorname{Pol}_{d-1}(G)) \overset{\sim}{\To}\PolCohom{d}{\bullet}{}{G}{\mathbb{R}}$, given on continuous cochains by
\begin{equation}
(\tau^n\xi)(g_0,\dots ,g_n) = \xi(g_0,\dots ,g_n)(\bbb),\notag
\end{equation}
and with inverse defined, also at the level of continuous cochains, by
\begin{equation}
(\tau^n)^{-1}(\xi)(g_0,\dots ,g_n)(t) = \xi( t^{-1}g_0,\dots ,t^{-1}g_n ). \notag
\end{equation}

\item More generally, let $G_1=G_2=G$ 
 and let  $\mathcal{E}$ be a continuous $G$-module. Considering $C(G,\E)$ as a $G_1\times G_2$-module with the action $(g_1,g_2).f(g) := g_1.f(g_1^{-1}gg_2)$ there is 

 an  isomorphism
\begin{equation}\label{linear-description-iso-2}
\chi^{\bullet} \colon \Cohom^{\bullet}\big(G_1,C(G,\mathcal{E})^{G_2(d)}\big) \overset{\sim}{\To} 
\Cohom_{(d)}^\bullet \left(G_2, C(G,\E)^{G_1} \right)= \PolCohom{d}{\bullet}{}{G}{\mathcal{E}}.
\end{equation}

\end{enumerate}
\end{proposition}

\begin{proof}
We first prove (ii). For the sake of clarity, denote by  $X$ a third copy of $G$ and write the coefficient module as $C(X,\E)$. Now define two complexes $(C^n, d_C^n)_{n\in \NN_0}$ and $(D^n, d_D^n)_{n\in \NN_0}$ of $G_1\times G_2$-modules as follows:
\[
C^n:=C(G_1^{n+1}, C(X,\E)) \ \text{ and } \ D^n:=C(G_2^{n+1}, C(X,\E)),
\]
with $G_1\times G_2$ actions
\begin{align*}
((g,h).f)(g_0,\dots, g_n)(x)&:= g(f(g^{-1}g_0,\dots, g^{-1}g_n )(g^{-1}xh)), \ f\in C^n;\\
((g,h).f)(h_0,\dots, h_n)(x)&:=g(f(h^{-1}h_0,\dots, h^{-1}h_n )(g^{-1}xh)), \ f\in D^n.
\end{align*}
In both cases, the coboundary maps are the standard inhomogeneous ones; i.e.~
\[
d_C^n(f)(g_0,\dots, g_{n+1})=\sum_{i=0}^{n+1} (-1)^i f(g_0,\dots, \hat{g}_i, \dots, g_{n+1}),
\]
and similarly for $d^n_D$. Augmenting $(C^\bullet, d_C^\bullet)$ with $\varps_C\colon C(X,\E)\to C^0$ given by $\varps_C(\xi)(g):=\xi$ and similarly for $(D^\bullet,d_D^\bullet)$ we obtain two complexes of $G_1\times G_2$-modules:
\begin{align}
0 &\to C(X,\E) \to (C^\bullet, d_C^\bullet)\label{C-reso}\\
0 &\to C(X,\E) \to (D^\bullet, d_D^\bullet)\label{D-reso}
\end{align}
Considering $C^n$ as a $G_2$-module, the action is only from the right in the $X$-variable, so
\[
(C^n)^{G_2(d)}=C\left(G_1^{n+1}, C(X,\E)^{G_2(d)}\right).
\]
Thus, as a complex of $G_1$-modules, $((C^n)^{G_2(d)}, d^n_C\restriction)$ is exactly the standard, relative injective resolution of the $G_1$-module $C(X,\E)^{G_2(d)}$, and we have therefore proved:
\begin{claim}\label{claim:first-claim}
Upon passing to $G_1$-invariants and cohomology, the complex $((C^n)^{G_2(d)}, d^n_C\restriction)$  computes $\Cohom^n(G_1, C(X,\E)^{G_2(d)})$. 
\end{claim}
\noindent Similarly, passing to $G_1$-invariants in the \eqref{D-reso} we see that 
\[
(D^n)^{G_1}=C(G_2^{n+1}, C(X,\E)^{G_1}),
\]
and hence, as a complex of $G_2$-modules,  $((D^\bullet)^G, d_D^\bullet\restriction)$ is the standard,  relatively injective resolution of the $G_2$-module $C(X,\E)^{G_1}$. The latter $G_2$-module identifies with $\E$ (as a $G$-module) via the map $\alpha\colon \E \to C(X,\E)^{G_1}$ given by $\alpha(\xi)(x):=x\xi$. 
This proves:
\begin{claim}\label{claim:second-claim}
Upon passing to $G_2(d)$-invariants and cohomology, the complex $((D^n)^{G_1}, d^n_D\restriction)$ computes $\Cohom^n_{(d)}(G, \E)$.
\end{claim}
\noindent Lastly we want to relate the two complexes; this is done by means of
\begin{claim}\label{claim:third-claim}
The map $\kappa^n\colon C^n \to D^n$ given by $\kappa^n(f)(h_0,\dots, h_n)(x):=f(xh_0,\dots, x h_n)(x)$ is an isomorphism of $G_1\times G_2$-complexes.
\end{claim}
\noindent To see this, we first note that a direct computation shows that $\kappa^\bullet$ is indeed a map of complexes commuting with the $G_1\times G_2$-actions, and that the map $(\kappa^n)^{-1}\colon D^n\to C^n$ given by $(\kappa^n)^{-1}(f)(g_0,\dots, g_n)(x)=f(x^{-1}g_0,\dots, x^{-1}g_n)(x)$ is its inverse. 
Thus, by composition we get an isomorphism: 
\begin{align*}
\tau^n\colon  \Cohom^n\left(G_1, C(X,\E)^{G_2(d)}\right)&= \Cohom^n\left(\left((C^\bullet)^{G_2(d)}\right)^{G_1},d^\bullet_C\restriction\right) \tag{Claim \ref{claim:first-claim}} \\
& \overset{\kappa^\bullet}{\simeq}\Cohom^n\left(\left((D^\bullet)^{G_2(d)}\right)^{G_1},d^\bullet_D\restriction\right) \tag{Claim \ref{claim:third-claim}}  \\
&= \Cohom^n\left(\left((C^\bullet)^{G_1}\right)^{G_2(d)},d^\bullet_D\restriction\right) \\
&=\Cohom^n_{(d)}(G,\E). \tag{Claim \ref{claim:second-claim}}
\end{align*}
This proves (ii), and to obtain (i) we simply put $\E=\RR$ and note that for an inhomogeneous cochain
\[
\xi\colon G^{n+1} \to \Pol_{d-1}(G)=C(X,\RR)^{G_2(d)}
\] 
the class $\tau^n([\xi])$ is represented by the cocycle 
\[
\ev_\bbb\circ \kappa(\xi)(h_0,\dots, h_n)=\xi(h_0,\dots, h_n)(\bbb).
\]
Conversely, for an inhomogeneous polynomial cocycle $\xi\colon G^{n+1} \to \RR=C(X,\RR)^{G_2}$  the image under $(\tau^n)^{-1}$ is represented by the inhomogeneous cocycle 
\[
(\kappa^n)^{-1} (\xi)\in C\left(G_1^{n+1}, C(X,\RR)^{G_2(d)}\right)=C\left(G^{n+1}, \Pol_{d-1}(G)\right)
\]
given by
$
(\kappa^n)^{-1} (\xi)(g_0,\dots, g_n) (x)= \xi(x^{-1}g_0,\dots, x^{-1}g_n)(x).
$
\end{proof}

\subsection{Inhomogeneous polynomial 1-cocycles}
In this section we give a different picture of polynomial 1-cohomology, analogous to the picture of ordinary cohomology in terms of inhomogeneous cocycles. 
This, in turn, will allow us to describe the first polynomial cohomology with trivial coefficients concretely in terms of polynomial maps, which is done in Section \ref{sec:polynomial_maps_on_groups} below.\\

Let $G$ be a lcsc group, $\mathcal{E}$ a continuous $G$-module, and consider the standard relatively injective resolution introduced in Section \ref{sec:Polynomial_cohomology}.
\begin{equation}
\xymatrix{ 0\ar[r] & \mathcal{E} \ar[r]^<<<<{d^{-1}} & C(G,\mathcal{E}) \ar[r]^{d^0} & C(G^2,\mathcal{E}) \ar[r]^{d^1} & C(G^3,\mathcal{E}) \ar[r] & \cdots }.\notag
\end{equation}
For functions $\xi\colon G\rightarrow \mathcal{E}$ we  define  the \emph{unitized difference operator} (see also Definition \ref{def:differential_notation} for notation) by
\begin{equation}
(\bar{\ldiff}_g\xi)(h) := (\ldiff_g\xi)(h) - (\ldiff_g\xi)(\bbb).\notag
\end{equation}

\begin{lem}\label{lem:diff-og-diff-bar}
For $\xi\in C(G,\E)$  and $d\in \NN$ the following are equivalent:
\begin{itemize}
\item[(i)] For all $g_1,\dots, g_d$ one has $ \bar{\del}_{g_1}\circ\cdots \circ \bar{\del}_{g_d}\xi=0 $
\item[(ii)] For all $g_1,\dots, g_d$ one has that $ {\del}_{g_1}\circ\cdots \circ {\del}_{g_d}\xi $ is a constant function into $\E$.
\end{itemize}
\end{lem}
\begin{proof}
For $d=1$, we have $\bar{\del}_{g_1}\xi:=\del_{g_1}\xi - \del_{g_1}\xi(\bbb)$, so if $\bar{\del}_{g_1}\xi=0$ then clearly $\del_{g_1}\xi$ is constant,
and, conversely, if $\del_{g_1}\xi$ is constant then it equals ${\del}_{g_{1}}\xi(\bbb)$, so $\bar{\del}_{g_1}(\xi)=0$. For the general case, one first observes that
\begin{align*}
\del_{g_1}\circ \cdots \circ \del_{g_d}(\xi) &= \del_{g_1}\circ \cdots \circ \del_{g_{d-1}}\left(\del_{g_d}(\xi)-\del_{g_d}(\xi)(\bbb) +\del_{g_d}(\xi)(\bbb)   \right)\\
&= \del_{g_1}\circ \cdots \circ \del_{g_{d-1}}(\bar{\del}_{g_d}(\xi)) + \underbrace{\del_{g_1}\circ \cdots \circ \del_{g_{d-1}}(  \del_{g_d}(\xi)(\bbb) )}_{\text{constant as function into $\E$}}  \\
\end{align*}
and by iterating this argument we see that $\ldiff_{g_1}\circ \cdots \ldiff_{g_n}(\xi)$ and  $\bar{\ldiff}_{g_1}\circ \cdots \circ \bar{\ldiff}_{g_d}(\xi)$ differ by a constant function. Thus, if 
$\bar{\ldiff}_{g_1}\circ \cdots \circ \bar{\ldiff}_{g_d}(\xi) =0$ then  $\ldiff_{g_1}\circ \cdots \ldiff_{g_n}(\xi)$ is constant and, conversely, if $\ldiff_{g_1}\circ \cdots \ldiff_{g_n}(\xi)$ is constant then so is $\bar{\ldiff}_{g_1}\circ \cdots \circ \bar{\ldiff}_{g_d}(\xi)$, and since the latter function is normalized to be 0 at $\bbb$, it follows that $\bar{\ldiff}_{g_1}\circ \cdots \circ \bar{\ldiff}_{g_d}(\xi)=0$.
\end{proof}

Note that when the $G$-action on $\E$ is trivial,  condition (ii) in Lemma \ref{lem:diff-og-diff-bar} is equivalent to $\xi\in C(G,\E)^{G(d+1)}$. We will also need a bit of information regarding the kernel of $d^1$. To this end, note that  $d^1(\xi)(s,t,u):=\xi(t,u)-\xi(s,u)+\xi(s,t)$ so $d^1(\xi)=0$ implies
\begin{equation}\label{d1=0-eq}
\xi(s,u)=\xi(s,t)+\xi(t,u).
\end{equation}
Using this, it easily follows that for $\xi\in \ker(d^1)\subset C(G^2,\E)$  we have
\begin{align}
\xi(\bbb,\bbb)&=0 \label{cocycle-eq-1}\\
\xi(g,h)&= \xi(g,\bbb)+\xi(\bbb,h)\label{cocycle-eq-2}\\
\xi(\bbb,g)&=-\xi(g,\bbb)\label{cocycle-eq-3}
\end{align}
In order to give an inhomogeneous picture of polynomial cohomology we need a bit of notation.
\begin{notation}\label{not:beta}
For $\xi\in C(G^2,\mathcal{E})$ we denote by $\bar{\xi} \in C(G,\mathcal{E})$ the map $\bar{\xi}(g) := \xi(\bbb,g)$ and by $\beta\colon C(G^2,\mathcal{E}) \to C(G,\mathcal{E})$ the  map $\xi \mapsto \bar{\xi}$.
\end{notation}
The following proposition now generalizes the usual description of cohomology in terms of inhomogeneous $1$-cocycles; 

\begin{prop}\label{prop:1-cocycle_description}
The map $\beta\colon C(G^2,\E)\to C(G,\E)$ restricts to a continuous bijection from  $\operatorname{Z}^1_{(d)}(G,\E):=C(G^2,\E)^{G(d)}\cap \ker(d^1)$ onto
\[
\mathcal{P}:=\{\eta\in C(G, \E) \mid \eta(\bbb)=0 \text{ and } \bar{\del}_{g_1}\circ \cdots \circ \bar{\del}_{g_d}\eta=0 \text{ for all } g_1,\dots, g_d\in G\}.
\]
\end{prop}
\begin{proof}
Since the topology on $C(G^2,\E)$  and $C(G,\E)$ is given by uniform convergence on compact subsets the continuity of $\beta$ is clear. It is furthermore injective on $\ker(d_1)$, because if $d^1(\xi)=0$ and $\bar{\xi}=0$ then $\xi(\bbb,g)=0$  and, by \eqref{cocycle-eq-3}, also $\xi(g,\bbb)=0$ for all $g\in G$. Thus, by \eqref{cocycle-eq-2}, $\xi(g,h)=\xi(g,\bbb)+\xi(\bbb,h)=0$. We now need to prove that $\beta$ takes values in the prescribed set. So, let $\xi\in \operatorname{Z}^1_{(d)}(G,\E)$ and $g_1,\dots, g_d\in G$ be given. Since $\xi(\bbb,\bbb)=0$, we have $\bar{\xi}(\bbb)=0$ so we only need to prove that
$
\bar{\del}_{g_1}\circ \cdots\circ \bar{\del}_{g_d}\bar{\xi}=0.
$
Using the three basic cocycle properties above, we now get:
\begin{align*}
\overline{\del_g\xi}(h) &:= \del_{g}\xi(\bbb,h)= g\xi(g^{-1},g^{-1}h)-\xi(\bbb,h)=\\
&= g\Big(\xi(g^{-1},\bbb)+\xi(\bbb,g^{-1}h)  \Big)-\xi(\bbb,h)\\
&=g\xi(\bbb,g^{-1}h)-\xi(\bbb,h) - g\xi(\bbb,g^{-1}) +\xi(\bbb,\bbb)\\
&=\del_g\bar{\xi}(h) -\del_g\bar{\xi}(\bbb)=\bar{\del}_g(\bar{\xi})(h).
\end{align*}
Inductively we therefore get  that
\begin{align}\label{bar-compatibility}
\overline{\del_{g_1}\circ \cdots \circ \del_{g_d}(\xi)}=\bar{\del}_{g_1}\circ \cdots \bar{\del}_{g_d}(\bar{\xi}),
\end{align}
for $\xi \in \ker(d^1)$. 
Thus, if $\xi\in \operatorname{Z}^1_{(d)}(G,\E)$ then the left hand side of  \eqref{bar-compatibility} vanishes and hence so does the right hand side; i.e.~$\bar{\xi}\in \mathcal{P}$.  To prove that $\beta$ is surjective, let $\eta\in \mathcal{P}$ be given and set $\tilde{\eta}(g,h):= \eta(h)-\eta(g)$. Then clearly $\tilde{\eta}\in C(G^2,\E)$ and a direct computation shows that $d^1\tilde{\eta}=0$. Since $\eta(\bbb)=0$, it is furthermore clear that $\bar{\tilde{\eta}}=\eta$ so all we have to prove is that $\tilde{\eta}\in C(G,\E)^{G(d)}$. Since we have already established that $\tilde{\eta}\in \ker(d^1)$ we may use \eqref{bar-compatibility} to conclude that
$
\overline{\del_{g_1}\circ \cdots \circ \del_{g_d}(\tilde{\eta})}=0.
$
However, as $\tilde{\eta}\in \ker(d^1)$ and $\ker(d^1)$ is a $G$-invariant subspace, also $\del_{g_1}\circ \cdots \circ \del_{g_d}(\tilde{\eta})\in \ker(d^1)$, and since $\beta$ is injective on $\ker(d^1)$ we conclude that $\del_{g_1}\circ \cdots \circ \del_{g_d}(\tilde{\eta})=0$ as desired.
\end{proof}
\color{black}

\begin{example}[quadratic 1-cocycles] \label{ex:quadratic_desc}
By the Lemma \ref{lem:diff-og-diff-bar} and Proposition \ref{prop:1-cocycle_description}, we may describe the inhomogeneous ``quadratic'' 1-cocycles $\xi\colon G\rightarrow \mathcal{E}$ as precisely those unital maps for which, for all $g,h\in G$, $({\ldiff}_g\circ {\ldiff}_h) \xi$ is constant. Computing this we get
\begin{align*}
(\ldiff_g\circ\ldiff_h)(\xi)(k) & = gh.\xi((gh)^{-1}k) -g.\xi(g^{-1}k) - h.\xi(h^{-1}k) + \xi(k) \\
 & = gh.\xi((gh)^{-1}) -g.\xi(g^{-1}) - h.\xi(h^{-1}),
\end{align*}
where the second equality follows by letting $k=\bbb$. This can be rewritten as
\begin{equation} \label{eq:quadratic_desc}
\xi(ghk) = \xi(gh) + g.\xi(hk) + ghg^{-1}.\xi(gk) - ghg^{-1}.\xi(g) - g.\xi(h) - gh.\xi(k), \quad g,h,k\in G.
\end{equation}
\end{example}

\section{Polynomial maps on groups} \label{sec:polynomial_maps_on_groups}
In this section we study the space of polynomials on a group, which was already ad hoc introduced in the previous section,  and show  that $\Cohom^1_{(d)}(G,\RR)$ can be described directly in terms of the polynomials on $G$.
As already mentioned in the introduction, the (abstract) notion of polynomial maps on groups goes  back (at least) to the work of Passi \cite{Passi-I,Passi-II, Passi-functor}  and has since then appeared in a number of different contexts; see e.g. \cite{ Babakhanian:Cohomological_Methods, Buckley-1, Leibman:Polynomial_mappings, Szekelyhidi} and references therein.
We now formally define the space of polynomial maps:

\begin{definition}[polynomial maps]\label{def:pol-maps}
Let $G$ be a lcsc group and let $\xi\in C(G,\mathbb{R}) \setminus\{0\}$. We say that $\xi$ is a \emph{polynomial} of degree at most $d\in \mathbb{\NN}_0$ if for all $g_1,\dots ,g_{d+1}\in G$ we have 
\begin{equation} \label{eq:polynomial_map_def}
(\ldiff_{g_1}\circ \cdots \circ \ldiff_{g_{d+1}})(\xi) = 0,
\end{equation}
where $C(G,\RR)$ is considered a $G$-module for the left regular action. The \emph{degree} $\deg \xi$ of a polynomial map $\xi$ is the smallest number $d$ such that $\xi$ satisfies \eqref{eq:polynomial_map_def} for all $g_1,\dots ,g_{d+1} \in G$. 
Moreover, the zero-map is formally included in the set of polynomials and assigned the degree $-\infty$. We denote the set of polynomials of degree at most $d$ by $\operatorname{Pol}_d(G) := C(G,\mathbb{R})^{G(d+1)}$ and by $\operatorname{Pol}(G)$ the set  $\cup_{d\in \mathbb{Z}_{\ast}} \operatorname{Pol}_d(G)$. Lastly, a polynomial $\xi$ is said to be \emph{unital} if $\xi(\bbb)=0$;  we denote the set of unital polynomials by $\Pol^0(G)$ and those of degree at most  $d$ by $\Pol_d^0(G)$.
\end{definition}

\begin{remark}\label{rem:right-diff}
 One could of course also define polynomials by means of right differences instead of left differences; i.e.
for a function $\xi\colon G \rightarrow 
 \mathbb{R}$,  consider  the \emph{right-difference operator} defined by 
\begin{equation}
(\rdiff_g \xi)(h) := \xi(hg) - \xi(h),\notag
\end{equation}
and introduce right polynomials accordingly. However, 
by  \cite[Corollary 2.13]{Leibman:Polynomial_mappings},  $\xi\in C(G,\mathbb{R})$ is a left polynomial map of degree $d$ if and only if  it is a right polynomial of degree $d$.

\end{remark}

\begin{rem}  We record the following basic facts concerning polynomials:
\begin{enumerate}
\item When $G$ is equal to $\RR$ (or more generally $\RR^n$) the above definition recovers the classical notion of polynomials and their degrees; we leave the argument as an exercise. 
\item The set $\Pol_0(G)$ consists of the constant functions on $G$ and  the set $\Pol_1(G)$ consists of functions of the type $\xi=\varphi +r$ where $\varphi\colon G\to \RR$ is a (continuous) homomorphism and $r\in \RR$ is a constant.

\item When $G$ is compact the only polynomials are the constant functions. For polynomials of degree 1 this is clear from the description just given, since the image of $G$ under a continuous homomorphism is a compact, additive subgroup of $\RR$ and hence equal to $\{0\}$. The general case now follows by induction on the degree. Note that this is a special case of the situation treated in Remark \ref{rem:compact-abelianization}.
\end{enumerate}
\end{rem}

In the language just introduced, Lemma \ref{lem:diff-og-diff-bar} and  Proposition \ref{prop:1-cocycle_description} simply say, that (with trivial coefficients) the space of  homogeneous $1$-cocycles of polynomial degree $d$ in the standard resolution is isomorphic to the space of  unital polynomial maps of degree at most $d$. The following proposition now describes (again for trivial coefficients) the space of polynomial $1$-coboundaries:

\begin{proposition} \label{prop:HdP1_R_description}
Let $G$ be a lcsc group and let $d\in \mathbb{N}$. Under the map $\beta\colon \xi\mapsto \bar{\xi}$ defined in Notation \ref{not:beta}, the set of polynomial 1-coboundaries $\operatorname{B}^1_{(d)}(G,\RR):=d^0(C(G,\mathbb{R})^{G(d)})$ maps bijectively onto the space of (continuous) unital polynomial maps of degree at most $d-1$; that is we obtain an isomorphism
\begin{equation}
\PolCohom{d}{1}{}{G}{\mathbb{R}} \cong \operatorname{Pol}_{d}(G)/\operatorname{Pol}_{d-1}(G).\notag
\end{equation}
In particular, $\PolCohom{d}{1}{}{G}{\mathbb{R}}$ is Hausdorff and the natural $G$-action on $\operatorname{Z}^1_{(d)}(G,\RR)$ induces  the trivial action on $\PolCohom{d}{1}{}{G}{\mathbb{R}}$.
\end{proposition}
We remark that all statements in Proposition \ref{prop:HdP1_R_description} are trivial when $d=1$, since $\Cohom^1(G,\RR)\simeq \Hom(G,\RR)$ and since the induced action is  trivial already  at the level of inhomogeneous cocycles (i.e.~$G$-invariant functions).

\begin{proof}
Since $\operatorname{B}^1_{(d)}(G,\RR):=d^0\left(C(G,\RR)^{G(d)}\right)$ is a subset of  $\operatorname{Z}^1_{(d)}(G,\RR)$, on which we know that the map $\beta$ is already injective and takes values in $\Pol_{d}^0(G)$, we have to prove that $\beta$ restricted to the coboundaries takes values in $\Pol_{d-1}^{0}(G)$ and is surjective onto this set. Recall that $d^0(\eta)(g,h):=\eta(h)-\eta(g)$, so if $\xi=d^0(\eta)$ for $\eta\in C(G,\RR)^{G(d)}$ then we have 
\begin{align}\label{beta-formula}
\beta(\xi)(g)=\beta(d^0\eta)(g)=d^0(\eta)(\bbb,g)=\eta(g)-\eta(\bbb).
\end{align}
Thus,
\[
\del_{g_1}\circ \cdots \del_{g_d}(\beta(\xi))= \del_{g_1}\circ \cdots \del_{g_d}(\eta-\eta(\bbb))=\del_{g_1}\circ \cdots \del_{g_d}(\eta)=0,
\]
and hence $\deg(\beta(\xi))\leq d-1$. On the other hand, given $\eta\in \Pol_{d-1}^0(G)\subset C(G,\RR)^{G(d)}$, the  computation \eqref{beta-formula} shows that $\beta(d^0\eta)=\eta$, 
and hence $\beta$ is surjective from $\operatorname{B}^1_{(d)}(G,\RR)$ onto $\Pol_{d-1}^0(G)$. Moreover since $\beta\colon \operatorname{Z}^1_{(d)}(G,\RR) \to \Pol_{d}(G)$ is  continuous, it follows from this that $\operatorname{B}^1_{(d)}(G,\RR)=\beta^{-1}(\Pol_{d-1}^0(G))$ is closed in $\operatorname{Z}^1_{(d)}(G,\RR)$  since $\Pol_{d-1}^0(G)$ is closed in $\Pol_d^0(G)$, and hence that $\Cohom_{(d)}^{1}(G,\RR)$ is Hausdorff. Moreover,  $\beta$ induces an isomorphism of topological vector spaces
\[
\Cohom_{(d)}^1(G,\RR)\simeq \Pol_{d}^0(G)/\Pol_{d-1}^0(G)\simeq \Pol_d(G)/\Pol_{d-1}(G)
\]
as claimed, where the latter isomorphism is induced by the (split) inclusion $\iota\colon \Pol^0(G) \to \Pol(G)$.
 The induced action on the right hand side is trivial, since for $\xi\in \Pol_{d}(G)$ and $g\in G$ we have $\del_g\xi\in \Pol_{d-1}(G)$. However, $\iota \circ \beta$ is not quite a $G$-equivariant map at the level of cocycles, but we now show that the induced map $\Cohom_{(d)}^1(G,\RR) \to \Pol_{d}(G)/\Pol_{d-1}(G)$ is. More precisely, we show that $\iota\circ\beta(g.f)-g.(\iota\circ \beta(f))$ differ by a constant map  --- in particular the difference is in $\Pol_{d-1}(G)$. This follows from \eqref{cocycle-eq-3}, \eqref{cocycle-eq-2} and \eqref{d1=0-eq}, since for $f\in \ker(d^1)$ we have
\[
(\iota \circ \beta(g.f)-g. (\iota\circ\beta(f))(x)= f(g^{-1},g^{-1}x)-f(\bbb,g^{-1}x)=f(g^{-1},g^{-1}x)+f(g^{-1}x,\bbb)= f(g^{-1},\bbb).
\]
\end{proof}

\begin{corollary} \label{cor:cohom_fin_dim_implies_pol_fin_dim}
Let $G$ be a cohomologically finite dimensional lcsc group. Then $ \operatorname{Pol}_d(G)$ is finite dimensional for all $d\in \mathbb{N}_0$.
\end{corollary}

\begin{proof}
By induction on $d$. For $d=0$ this is trivial, and for $d=1$ we observe that 
\[
\dim_{\mathbb{R}} \operatorname{Pol}_1(G) = 1+ \dim_{\mathbb{R}} \Cohom^1(G,\mathbb{R}) < \infty. 
\]
The inductive step follows from part (i) of Proposition \ref{prop:PolCohom_linear_description}.
\end{proof}

\color{black}

\begin{remark}\label{rem:explicit_description_of_tau}
We now return to the isomorphism $\tau^1\colon \Cohom^1(G,\Pol_{d-1}(G)) \to \Cohom^1_{(d)}(G,\RR)$ given by part (i) of Proposition \ref{prop:PolCohom_linear_description}, with the aim of providing a more explicit description  of this map in terms of the description of $\Cohom^1_{(d)}(G,\RR)$ by means of polynomials on $G$.  For notational convenience, we denote $(\tau^1)^{-1}$ by $\tau' $. 
We first describe $\tau^1$ at the level of inhomogeneous cocycles, i.e.~from $\operatorname{Z}^1(G,\Pol_{d-1}(G))$  to  $\Pol_d(G)$. Given an inhomogeneous 1-cocycle $c\colon G\to \Pol_{d-1}(G)$, then the standard map back to the homogeneous picture sends $c$ to $\tilde{c}\colon G\times G\to \Pol_{d-1}(G)$ given by $\tilde{c}(g_0,g_1):=g_0c(g_0^{-1}g_1)$ (see e.g.~\cite[I, n$^{\circ}$ 3.2 \& III, n$^{\circ}$ 1.3 ]{Guichardet:Cohomologie}) and thus 
\[
\tau^1(\tilde{c})(g_0,g_1)=(g_0c(g_0^{-1}g_1))(\bbb)=c(g_0^{-1}g_1)(g_0^{-1}).
\]
This is then a homogeneous polynomial 1-cocycle, and to get back to the inhomogeneous picture (i.e.~the description using polynomials) we need to apply the `bar-map' $\beta$ defined in Notation \ref{not:beta}.  That is, we fix the first variable $g_0=\bbb$ and obtain the map
$
g\mapsto c(g)(\bbb),
$
and this is then the polynomial in $\Pol_d(G)$ representing  $\tau^1([c])$ in $\Cohom_{(d)}^1(G,\RR)=\Pol_{d}(G)/\Pol_{d-1}(G)$. To get an explicit description of the inverse map $\tau'$, consider a polynomial $\xi\in \Pol_d(G)$. The corresponding homogeneous polynomial 1-cocycle is given by $\hat{\xi}(g_0,g_1):=\xi(g_1)-\xi(g_0)$ (cf.~the proof of Proposition \ref{prop:1-cocycle_description}) and therefore
\[
\tau'(\hat{\xi})(g_0,g_1)(t)= \hat{\xi}(t^{-1}g_0,t^{-1}g_1)=\xi(t^{-1}g_1)-\xi(t^{-1}g_0)
\]
The inhomogeneous 1-cocycle corresponding to  $\tau'([\hat{\xi}])$ is then obtained by fixing the variable $g_0=\bbb$, and hence $\tau'([\xi])$ is represented the 1-cocycle $c\colon G\to \Pol_{d-1}(G)$  given by
\[
c(g)(t)= \xi(t^{-1}g)-\xi(t^{-1})=\del_g(\tilde{\xi})(t),
\]
where $\tilde{\xi}(g):=\xi(g^{-1})$. That is, $c(g)=\del_g\tilde{\xi}$.

\end{remark}

\begin{cor}\label{cor:proj_induces_injection}
The map $\pi^*\colon \Cohom^1(G,\Pol_d(G))\to \Cohom^1(G,\Cohom_{(d)}^1(G,\RR))$, induced by the quotient map $ \pi\colon \Pol_{d}(G)\to \Cohom_{(d)}^1(G, \RR)$, is injective 
\end{cor}
\begin{proof}
Considering the short exact sequence 
\[
0\to \Pol_{d-1}(G) \to \Pol_{d}(G) \to \Cohom^1_{(d)}(G,\RR) \to 0,
\]
and the corresponding long exact sequence in cohomology, the statement is seen to be equivalent to showing that $\iota\colon \Pol_{d-1}(G) \to \Pol_{d}(G)$ induces the zero map in 1-cohomology.  To this end, we first prove that the following diagram commutes
\[
\xymatrix{
\Cohom^1(G,\Pol_{d-1}(G)) \ar[r]^{\iota^*} \ar[d]^{\tau^1}_{\simeq} &  \Cohom^1(G, \Pol_d(G)) \ar[d]^{\tau^1}_{\simeq} \\
\Cohom^1_{(d)}(G,\RR) \ar[r]_{(\subset)^*} & \Cohom^1_{(d+1)}(G,\RR) ,
}
\]
where $(\subset)^*$ is the map induced by the inclusion $\Pol_{d}(G)\subset \Pol_{d+1}(G)$ and $\tau^1$ is the isomorphism given by Proposition \ref{prop:PolCohom_linear_description}. For this, we will use the explicit description of $\tau^1$ and $\tau':=(\tau^{1})^{-1}$ at the level of cocycles discussed in Remark \ref{rem:explicit_description_of_tau}. Let $\eta\in \Pol_{d}(G)$ be given. Since everything depends only on the class $[\eta]\in \Cohom^1_{(d)}(G,\RR)$, by subtracting a constant polynomial we may assume that $\eta(\bbb)=0$. Now consider the cocycle $g\mapsto \del_g\tilde{\eta}$ representing $\tau'([\eta])\in \Cohom^1(G,\Pol_{d-1}(G))$, where $\tilde{\eta}(g):=\eta(g^{-1})$. Composing with $\iota$ just gives the same cocycle now considered as taking values in $ \Pol_{d}(G)$, and applying $\tau^1$ amounts to evaluating at $\bbb$; that is, $\tau^1\circ \iota^*\circ \tau'([\eta])\in \Cohom^1_{(d+1)}(G,\RR)$ is represented by the polynomial:
\[
g\mapsto (\del_{g}\tilde{\eta})(\bbb)= \tilde{\eta}(g^{-1})-\tilde{\eta}(\bbb)=\eta(g)-\eta(\bbb)=\eta(g).
\]
Hence the map $\tau^1\circ \iota^*\circ\tau'$ agrees with  $(\subset)^*\colon\Cohom^1_{(d)}(G,\RR) \to \Cohom^1_{(d+1)}(G,\RR)$, and the latter map is clearly zero. \end{proof}

\color{black}

\begin{remark} \label{rmk:ldiff_iso_cohom}
As follows from Remark \ref{rem:explicit_description_of_tau}, the map $\ldiff\colon C(G,\mathbb{R}) \rightarrow C(G, C(G,\mathbb{R}))$ given by $ f \mapsto (g \mapsto \ldiff_g f)$, when pre-composed with $\operatorname{inv}^* \colon C(G,\mathbb{R}) \rightarrow C(G,\mathbb{R}) \colon f \mapsto (g\mapsto f(g^{-1}))$, induces an isomorphism $\overline{\ldiff \circ \operatorname{inv}^*} \colon \operatorname{Pol}_d(G)/\operatorname{Pol}_{d-1}(G) \xrightarrow{\cong} \Cohom^1(G,\operatorname{Pol}_{d-1}(G))$.  However, observe that since $\operatorname{inv}^*$ induces a degree-preserving  linear automorphism of $\operatorname{Pol}(G)$, it follows that the map $\bar{\ldiff} \colon \operatorname{Pol}_d(G)/\operatorname{Pol}_{d-1}(G) \rightarrow \Cohom^1(G,\operatorname{Pol}_{d-1}(G))$, mapping $\xi \in \operatorname{Pol}_d(G)$ to the equivalence class of the cocycle $g\mapsto \ldiff_g \xi$, is an isomorphism as well. 
\end{remark}

The remainder of this section is devoted to a more detailed analysis of the degree function and its interplay with the differentiation operators,
and for this  we will  use the extended addition and subtraction on $\ZZ_*:=\NN_{0}\cup\{-\infty\}$ defined in the Section \ref{sec:notation}.

\begin{proposition} \label{prop:pol_diff_is_pol}
Let $G$ be a group and $\xi\colon G \rightarrow \mathbb{R}$ be a 
polynomial map of degree $ d \geq 1$. Then for every $s\in G$, the map $\varphi_{\xi,s}\colon g\mapsto (\ldiff_g \xi)(s)$ is a 
polynomial map of degree $ d$, and so is  $g\mapsto (\rdiff_g\xi)(s)$.
\end{proposition}
Here $\rdiff_g$ denotes the right difference operator introduced in Remark \ref{rem:right-diff}.

\begin{proof}
For any function $\xi\colon G\rightarrow \mathbb{R}$, a direct computation verifies that the differential satisfies
\begin{equation}\label{diff-wrt-product}
\ldiff_{gh}\xi = (\ldiff_g\circ \ldiff_h)(\xi) + \ldiff_g\xi + \ldiff_h\xi, \quad g,h\in G.
\end{equation}
Thus for any $h\in G$ we have
\begin{align}
(\rdiff_h\varphi_{\xi,s})(g) & = \varphi_{\xi,s}(gh)-\varphi_{\xi,s}(g) \notag \\
 & = (\ldiff_g \circ \ldiff_h)(\xi)(s) + (\ldiff_h \xi)(s) \notag \\
 & = \varphi_{\ldiff_h \xi,s}(g) + \varphi_{\xi,s}(h)\label{eq:diff-of-phi}.
\end{align}
If $d=1$ then $\varphi_{\ldiff_h \xi,s}=0$ so $(\rdiff_h\varphi_{\xi})$ is constant equal to $\varphi_{\xi,s}(h)$,  and hence $\deg(\varphi_{\xi,s})\leq 1$. However, since $d=1$, $\ldiff_h \xi$ is constant for every $h\in G$ and for some $h_0$ this constant is non-zero. So, $\rdiff_{h_0}\varphi_{\xi,s} \neq 0$ proving that $\deg(\varphi_{\xi,s})=1$.
The general case now follows by induction on $d$. For the inductive step, assume that the statement is true for $d-1$ and that $\xi$ has degree $d\geq 2$. For $h\in G$, $\del_h\xi$ has degree at most $d-1$, so the induction takes over and gives $\deg(\varphi_{\ldiff_h \xi,s})\leq d-1$ and since $\varphi_{\xi,s}(h)$ is constant in the variable $g$, $\deg(\varphi_{\xi,s})\leq d$ by the computation \eqref{eq:diff-of-phi}. But for some $h_0$, $\del_{h_0}\xi$ has degree equal to $d-1 \geq 1$ and hence, by the induction, so does  $\varphi_{\ldiff_{h_0} \xi,s}(-) + \varphi_{\xi,s}(h_0)= \rdiff_{h_0}\varphi_{\xi,s}$; thus, $\deg(\varphi_{\xi,s})=d$.
\end{proof}

Let $G$ be a lcsc group and let $\xi,\eta\colon G \rightarrow \mathbb{R}$ be polynomial maps on $G$. Then it is easy to see that the pointwise product $\xi\cdot \eta \colon g\mapsto \xi(g) \eta(g)$ is again a  polynomial map with $\deg (\xi\cdot \eta) \leq \deg \xi \plusdot \deg \eta$; indeed, we have

\begin{align}\label{eq:product_rdiff}
\rdiff_g (\xi\cdot \eta)(h) &= \xi(hg)\cdot (\rdiff_g\eta)(h) + (\rdiff_g \xi)(h) \cdot \eta(h) \ \textrm{and} \notag \\
\ldiff_g (\xi\cdot \eta)(h) &= \xi(g^{-1}h)\cdot (\ldiff_g\eta)(h) + (\ldiff_g \xi)(h) \cdot \eta(h),
 \end{align}
from (either of) which the inequality  follows by induction on $\deg\xi \plusdot \deg \eta$ (we shall actually show below that equality holds for connected, simply connected, nilpotent Lie  groups).
In particular the multiplication map induces a linear map
\begin{equation}
\PolCohom{d}{1}{}{G}{\mathbb{R}} \otimes \PolCohom{d'}{1}{}{G}{\mathbb{R}} \rightarrow \PolCohom{d+d'}{1}{}{G}{\mathbb{R}}\notag
\end{equation}
for each pair $(d,d') \in \mathbb{N}^2$. Note also, that  equation \eqref{eq:product_rdiff} implies that we have the following version of the Leibniz rule for the differentials:
\begin{equation}\label{eq:leibniz_rule}
\rdiff_g(\xi\eta)=\rdiff_g(\xi)\rdiff_g(\eta) + \rdiff_g(\xi)\eta + \xi\rdiff_g(\eta),
\end{equation}
and similarly for  $\ldiff$. \\

The last goal in this section is to give a sharper estimate on the degree of $\rdiff_g \xi$ for a polynomial map $\xi$. 
To this end, recall first that a \emph{central series} $\mathscr{G} = (G_{i})_{i\in \mathbb{N}}$ in a (topological) group $G$ is a decreasing sequence of (closed) normal subgroups $G_{i}\unlhd G$, with $G=G_{1}$ and such that $[G_{i},G] \subseteq G_{i+1}$ for all $i\in \mathbb{N}$.
The \emph{lower central series} of a group $G$ is the (decreasing) sequence $\mathscr{G}_{\text{min}} = (G_{[i]})_{i\in \NN}$ of subgroups of $G$ defined recursively by $G_{[1]} := G$ and $G_{[i+1]} := {[G,G_{[i]}]}$, where the latter denotes the group generated by commutators of elements from $G$ and $G_{[i]}$. 
In  case  $G$ is endowed with a topology, the lower central series is defined by closing up the algebraically defined ditto.
Observe that each $G_{[i]}$ is a characteristic subgroup of $G$; i.e.~globally preserved by any automorphism of $G$. Further, for any central series $\mathscr{G} = (G_i)_{i\in \NN}$ in $G$ we have, by construction, $G_{[i]} \leq G_i$ and moreover one may prove that $[G_{[i]},G_{[j]}]\leq G_{[i+j]}$ for all $i,j\in \NN$ (see e.g.~\cite[Corollary 0.31]{Baumslag:Nilpotent}).


\begin{definition}[degree wrt.~a central series]\label{def:degree-def}
Let $G$ be a (lcsc) group and $\mathscr{G}$ a central series in $G$ of finite length. For every $g\in G$ we define the degree $\deg_{\mathscr{G}} g$ of $g$ with respect to the central series $\mathscr{G}$ by
\begin{equation}
\deg_{\mathscr{G}} g := \max \{ i \mid g\in G_{i}, g\notin G_{i+1} \}.\notag
\end{equation}
\end{definition}
The following result now gives an improved bound on the degree of $\rdiff_g\xi$  in the situation where one knows where $g$ is located in the lower central series. 

\begin{lem}\label{lma:pol_diff_degree_wrt_lcs}
Let $G$ be a group and let $\xi\in \Pol(G)$. Then for $g\in G_{[k]}$ we have $\deg(\rdiff_g\xi)\leq \deg \xi \minusdot k$. Hence, $\deg(\rdiff_g\xi)\leq \deg \xi \minusdot \deg g$, when $\deg(g)$ is taken with respect to the lower central series.
\end{lem}

\begin{rem}
In the statement of Lemma \ref{lma:pol_diff_degree_wrt_lcs}, the group $G$ is not a priori assumed to carry a topology and the lower central series is therefore to be understood in the purely algebraic sense. Note, however, that if $G$ is a csc Lie group, then the algebraically  defined lower central series automatically consists of closed subgroups \cite[XII, Theorem 3.1]{Hochschild:Structure-of-Lie-grps} and hence, in this case, there is no difference between the topological and algebraic lower central series. More generally, if  $G$ is a lcsc group (possibly not of Lie type) and $\xi\in \Pol_{d}(G)$, then Lemma \ref{lma:pol_diff_degree_wrt_lcs} shows that $\rdiff_g \xi \in \Pol_{d\minusdot k}(G)$ for all $g\in G_{[k]}$ (the algebraically defined lower central series).  Moreover, the map $g\mapsto \rdiff_g\xi$ is continuous into  $C(G,\RR)$ (endowed with the Fr{\'e}chet topology of uniform convergence on compacts), and since $\Pol_{d\minusdot k}(G)$ is a closed subspace in $C(G, \RR)$, this shows that $\rdiff_g \xi \in \Pol_{d\minusdot k}(G)$ also for $g$ in the closure $\overline{G_{[k]}}$; i.e., the statement of Lemma \ref{lma:pol_diff_degree_wrt_lcs} holds true in the topological context as well.
\end{rem}

\begin{proof}[Proof of Lemma \ref{lma:pol_diff_degree_wrt_lcs}]
We prove the statement by induction on $k$. If $k=1$ then the statement is true by the definition of a polynomial map. 
 Assume now that the statement is true for $k-1\geq 1$ and let $x\in G_{[k]}$ be given. Assume first that $x=g^{-1}h^{-1}gh$ with $g\in G_{[k-1]}$ and $h\in G$. Then, as $G_{[k-1]}$ is normal in $G$, by computing modulo $\Pol_{d\minusdot k}(G)$ (symbolically represented by `$\equiv$') we get
\begin{align*}
\rdiff_x\xi &=\rdiff_{g^{-1}(h^{-1}gh)}(\xi)\\
&=\rdiff_{g^{-1}}\circ \rdiff_{h^{-1}gh}\xi + \rdiff_{g^{-1}}(\xi) +\rdiff_{h^{-1}g h}(\xi) \tag{by \eqref{diff-wrt-product}}\\
& \equiv  \rdiff_{g^{-1}}(\xi) +\rdiff_{h^{-1}g h}(\xi)\\
& \equiv  \rdiff_{g^{-1}}(\xi) +\rdiff_{h^{-1}} \circ \rdiff_{gh}(\xi) + \rdiff_{h^{-1}}(\xi) +\rdiff_{gh}(\xi)\\
& =  \rdiff_{g^{-1}}(\xi) +\rdiff_{h^{-1}}\Big(\rdiff_{g}\circ \rdiff_{h}(\xi) +\rdiff_{g}(\xi) +\rdiff_{h}(\xi) \Big) +\rdiff_{h^{-1}}(\xi) +\rdiff_{g}\circ \rdiff_{h}(\xi) +\rdiff_{g}(\xi)+\rdiff_{h}(\xi)\\
& \equiv \rdiff_{g^{-1}}(\xi) + \rdiff_{h^{-1}}\circ \rdiff_{h}(\xi)+\rdiff_{h^{-1}}(\xi) + \rdiff_{g}(\xi) + \rdiff_{h}(\xi)\\
&= \rdiff_{g^{-1}}(\xi) + \rdiff_{hh^{-1}}(\xi) +\rdiff_{g}(\xi)\\
&= \rdiff_{gg^{-1}}(\xi) -\rdiff_{g}\circ\rdiff_{g^{-1}}(\xi) \\
&=-\rdiff_{g}\circ\rdiff_{g^{-1}}(\xi)\equiv 0.
\end{align*} 
A completely analogous computation shows that  also $\rdiff_{x^{-1}}\xi\equiv 0$ and from \eqref{diff-wrt-product}  if follows that $\rdiff_x\xi\equiv \rdiff_y(\xi)\equiv 0$ implies that $\rdiff_{xy} \xi \equiv 0$. Hence $\rdiff_z(\xi)\equiv 0$ for all $z\in G_{[k]}$ as desired.
\end{proof}

As a consequence of Lemma \ref{lma:pol_diff_degree_wrt_lcs}, we also record the following result due to Leibman:
\begin{cor}[{\cite[Lemma 2.12 \& 2.14]{Leibman:Polynomial_mappings}}]\label{cor:leibman-cor}
If $\xi \in \Pol_{d}^0(G)$ then $\xi$ vanishes on $G_{[d+1]}$.
\end{cor}
\begin{proof}
For $g\in G_{[d+1]}$ we have $\deg(\rdiff_g\xi)\leq d\minusdot (d+1)=-\infty $ so $\rdiff_g\xi=0$. Thus
\[
0= \rdiff_g(\xi)(\bbb)=\xi(g)-\xi(\bbb)=\xi(g). \qedhere
\]
\end{proof}

We end this section with a small lemma to be used in the section to follow.
\begin{lem}\label{lem:poly-on-quotient}
Let $G$ be a lcsc group and let $Z\leq G$ be a normal subgroup isomorphic to $\RR$. If  $z_0\in Z\setminus \{\bbb_G\}$ and $\xi\in \Pol(G)$ satisfies $\rdiff_{z_0}\xi=0$ then $\xi$ descends to a polynomial $\bar{\xi}$ on  $G/Z$ of the same degree.
\end{lem}
\begin{proof}
To see that $\bar{\xi}$ is well defined we need to show that $\xi$ is constant on the cosets of $Z$. For $g\in G$, the left translate $g^{-1}.\xi$ is again a polynomial and 
hence so is the restriction $\eta:=(g^{-1}.\xi)\restriction_Z$. By assumption we have $\xi(hz_0)=\xi(h)$ for all $h\in G$ and hence $\xi(gz_0^2)=\xi(gz_0)=\xi(g)$ and, recursively, 
$
\xi(gz_0^n)=\xi(g)
$
for all $n\in \NN$. The map $\eta$ is therefore a polynomial on $Z\simeq \RR$  which is constant on an infinite set, and since the polynomials on $\RR$ are exactly the classical polynomials
this can only happen if $\eta=0$. That is,  $\xi(gz)=\xi(g)$ for all $z\in Z$ and therefore $\bar{\xi}\colon G/Z\to \RR$ is well defined. We furthermore have
\[
\rdiff_{z_0} \circ\ldiff_g(\xi) = \ldiff_g \circ \rdiff_{z_0}(\xi)=0 \text{ for all } g\in G,
\]
and hence $\overline{\del_g\xi}$ is well defined as well, and a direct computation verifies that
\[
\del_{\bar{g}_1}\circ \dots \circ \del_{\bar{g}_d} \bar{\xi} =\overline{\del_{{g}_1}\circ \dots \circ \del_{{g}_d} \xi},
\]
from which it follows that $\bar{\xi}$ is a polynomial of degree  $\deg(\xi)$. 
\end{proof}
\begin{rem}\label{rem:normal-subgroup-poly}
The proof of Lemma \ref{lem:poly-on-quotient} also shows the following general fact: if $G$ is a group and $H\leq G$ is a normal subgroup, then any  $\xi\in \Pol^0(G)$ with the property that $\rdiff_h\xi=0$ for all $h\in H$ descends to a polynomial $\bar{\xi}\in \Pol^0(G/H)$ of the same degree.
\end{rem}

\section{Nilpotent groups and their cohomology} \label{sec:Malcev}
In this section we collect the necessary prerequisites concerning nilpotent groups and their cohomology.
 For general background on nilpotent groups we refer to \cite{Kirillov:Introduction,Corwin-Greenleaf}.\\

Recall first that a group $G$ is called \emph{nilpotent} if $G_{[d]} = \{\bbb\}$ for some $d\in \mathbb{N}$, where  $G_{[d]}$ denotes the $d$'th group in the lower central series (see e.g.~the remarks preceding Definition \ref{def:degree-def} for more details);
in this case the \emph{(nilpotency) class} of $G$ is defined as the number $\operatorname{cl}(G) := \min \{ d \mid G_{[d]} = \{\bbb\} \} - 1$. Note that in the special case where $G$ is a connected, simply connected, nilpotent  Lie group, the algebraically defined lower central series automatically consists of closed subgroups  \cite[XII, Theorem 3.1]{Hochschild:Structure-of-Lie-grps}.\\

Secondly, recall that for any (real) Lie algebra $\mathfrak{g}$, the lower central series is defined (analogously to the definition for groups) by $\mathfrak{g}_{[i+1]}= \operatorname{span}_{\RR} [\mathfrak{g},\mathfrak{g}_{[i]}]$. Let $G$ be a connected, simply connected (henceforth abbreviated `csc'), nilpotent Lie group, and denote its Lie algebra $\mathfrak{g}$. Then for each $i$, one has that $G_{[i]}$ is a Lie subgroup of $G$ with Lie algebra $\mathfrak{g}_{[i]}$. Moreover,  for such $G$, the exponential map $\exp\colon \mathfrak{g} \rightarrow G$ is a \emph{global} diffeomorphism onto $G$, and it therefore also induces a diffeomorphism $\mathfrak{g}_{[i]} / \mathfrak{g}_{[i+1]} \rightarrow G_{[i]} / G_{[i+1]}$ for each each $i$. A (strong) \emph{Mal'cev basis} for $\mathfrak{g}$ (with respect to the lower central series) is a linear basis $( X_{i,j} )_{(i,j)\in\Bultiindex{B}{\multirank(G)}}$ of $\mathfrak{g}$, such that for each $i$, $X_{i,j} \in \mathfrak{g}_{[i]}$ for all $j$, and the set $\{X_{i,j}\}_j$ projects to a linear basis of $\mathfrak{g}_{[i]}/\mathfrak{g}_{[i+1]}$ (see Section \ref{sec:notation} for a definition of the multiindex $\Bultiindex{B}{\multirank(G)}$). Such a basis always exists \cite[Section 1.1]{Corwin-Greenleaf}  and once a Mal'cev basis is chosen,   the map
\begin{equation} \label{eq:Malcev_coordinates_def}
\mathfrak{g} \owns \sum_{(i,j)\in \Bultiindex{B}{\multirank(G)}} t_{i,j}X_{i,j} \longmapsto \prod_{(i,j)\in \Bultiindex{B}{\multirank(G)}} \exp(t_{i,j}X_{i,j})\in G,
\end{equation}
is a diffeomorphism as well \cite[Section 1.2]{Corwin-Greenleaf}, and the induced global coordinate system on $G$ is called (the system of) \emph{Mal'cev coordinates} relative to the chosen Mal'cev basis. Abusing terminology slightly, we will therefore also refer to the family $\{g_{i,j} := \exp(X_{i,j}) \mid (i,j) \in \Bultiindex{B}{\multirank(G)} \}$ as a Mal'cev basis of $G$ and denote $\exp(tX_{ij})$ by $g_{ij}^t$ so that each element  $g\in G$ can be uniquely written as
\[
g=\prod_{(i,j)\in \Bultiindex{B}{\multirank(G)}} g_{ij}^{t_{i,j}}, \qquad t_{ij}\in \RR.
\]
Here, and above, we use the ordered product notation introduced in the Section \ref{sec:notation}. Given any Mal'cev basis $(X_{i,j})_{(i,j)\in \Bultiindex{B}{\multirank(G)}}$ of $\mathfrak{g}$, for all $(i,j),(s,t)\in \Bultiindex{B}{\multirank(G)}$ the $ c^{i,j,s,t}_{k,l}\in \RR$ such that
\begin{equation}
[X_{i,j},X_{s,t}] = \sum_{(k,l)\in \Bultiindex{B}{\multirank(G)}} c^{i,j,s,t}_{k,l} X_{k,l}.\notag
\end{equation}
are called the structure constants of $\mathfrak{g}$ (with respect to the chosen basis), and in his groundbreaking paper \cite{Mal'cev}, Mal'cev proved the following result:

\begin{theorem}[Mal'cev] \label{thm:Malcev_rational}
A csc nilpotent Lie group $G$ has a lattice if and only if it has a Mal'cev basis with rational structure constants. Furthermore, every lattice $\Gamma$ in $G$ is cocompact and there exists a Mal'cev basis $(X_{ij})_{(i,j)\in \Bultiindex{B}{\multirank(G)}}$ which is based in $\Gamma$, in the sense that 
\[
\Gamma=\left\{ \prod_{(i,j)\in  \Bultiindex{B}{\multirank(G)}} g_{ij}^{m_{ij}}\  \Big{\rvert} \ m_{ij}\in \ZZ  \right\}.
\]
\end{theorem}

Any lattice in a csc nilpotent Lie group is necessarily torsion-free and finitely generated, and Mal'cev also proved  that the converse is true:

\begin{theorem}[Mal'cev]\label{thm:malcev-1}
Let $\Gamma$ be a finitely generated, torsion-free (discrete) nilpotent group. Then there exists a csc nilpotent Lie group $G$ such that $\Gamma$ embeds as a lattice in $G$. Furthermore, the embedding is unique up to natural isomorphism; that is, given any two such embeddings $i\colon \Gamma \rightarrow G$ and $j\colon \Gamma \rightarrow H$, there is an isomorphism $\psi\colon G \rightarrow H$ intertwining $i$ and $j$.
\end{theorem}

The csc nilpotent Lie group $G$ in Theorem \ref{thm:malcev-1} is called the \emph{Mal'cev completion} of $\Gamma$, and is occasionally denoted $ \Gamma \otimes \mathbb{R}$. For a proof of the theorem see \cite{Mal'cev}, or for an alternative approach \cite{Baumslag:Nilpotent} (which, in turn,  is based on \cite{Jennings:Residual_nilpotence}). The above motivates the following:

\begin{definition}[Mal'cev group]
Let $G$ be a locally compact, compactly generated topological group. We will say that $G$ is a \emph{Mal'cev group} if it satisfies either of the following two equivalent criteria:
\begin{enumerate}[(i)]
\item $G$ embeds as a closed, cocompact subgroup in a csc nilpotent Lie group,

\item $G$ is a torsion-free, nilpotent Lie group.
\end{enumerate}
If $G$ is a Mal'cev group then the csc nilpotent Lie group into which it embeds cocompactly is uniquely determined up to isomorphism and is, in analogy with the discrete case, called the (real) \emph{Mal'cev completion} of $G$ and denoted $G\tens \RR$. If $G$ is a Mal'cev group then $G_i:=G\cap (G\tens \RR)_{[i]}$ defines a central series in $G$ which we will refer to as the \emph{Mal'cev central series}. 
\end{definition}
\begin{rem}
Since csc nilpotent Lie groups are torsion-free, the implication from (i) to (ii) is clear and the fact that (ii) implies (i) is due to Mal'cev in the case when $G$ is discrete and 
Wang in general; see comments right before Proposition 4.6 in \cite{Wang:Discrete_subgroups} or \cite[Theorem 2.20]{Raghunathan}. 
\end{rem}

We will need the following additional facts about the class of Mal'cev groups.
\begin{enumerate}

\item Reasoning exactly as in the discrete case (see e.g.~\cite[Chapter 5]{Corwin-Greenleaf}) one may prove that any Mal'cev group $G$ admits a Mal'cev basis for $G\tens \RR$ (strongly) based in $G$; that is, there exists a Mal'cev basis $(X_{ij})_{(i,j)\in  \Bultiindex{B}{\multirank(G\tens \RR)}}$ such that
\[
G=\left\{ \prod_{(i,j)\in  \Bultiindex{B}{\multirank(G\tens \RR)}} \exp(m_{ij}X_{ij})\  \Big{\rvert} \ m_{ij}\in Z_{ij}  \right\},
\]
where the sets $Z_{ij}\subset \RR$ are equal to either $\ZZ$ or $\RR$. Abusing notation slightly,  we will refer to the elements $g_{ij}:=\exp(X_{ij})$ as a Mal'cev basis for $G$. 
\item
By fixing a Mal'cev basis for $G$ we also obtain  isomorphisms of abelian groups
\[
G_i/G_{i+1} \simeq \oplus_{j=1}^{n_i} Z_{ij} \simeq \ZZ^{m_i} \oplus\RR^{m_i'},
\]
where $m_i,m_i'\in \NN_0$ sum up to $n_i:=\dim(\mfg_{[i]}/\mfg_{[i+1]})$;  here $\mfg$ denotes the Lie algebra of $G\tens \RR$ and $(G_i)_i$ is the Mal'cev central series defined above.

\item Since  a Mal'cev group $G$ is nilpotent, it always has non-trivial center, and upon choosing a Mal'cev basis for $G$, one  can always find a central subgroup $Z$ such that $Z$ is of the form $Z:=\{g_{\text{cl}(G),j_0}^m\mid m\in Z_{\text{cl}(G),j_0} \}$, where $Z_{\text{cl}(G),j_0}$ is either $\ZZ$ or $\RR$ and such that $G/Z$ is again a Mal'cev group with Mal'cev basis $(\bar{g}_{ij})_{(i,j)\neq (\text{cl}(G),j_0)}$. 
In particular, we get a natural, continuous cross section $\sigma\colon G/Z \to G$ of the quotient homomorphism by setting
\[
\sigma \colon \prod_{(i,j)\neq (\text{cl}(G),j_0)}\bar{g}_{ij}^{t_{ij}} \longmapsto \prod_{(i,j)\neq (\text{cl}(G),j_0)}g_{ij}^{t_{ij}}
\]
This will be of importance in the sequel, as it is a necessary requirement for using the Hochschild-Serre spectral sequence in group cohomology \cite[III,  n$^\text{o}$ 5.1]{Guichardet:Cohomologie}.
\end{enumerate}

\begin{definition}[length and rank]
Let $G$ be a 
Mal'cev group. We denote the length of the (Mal'cev-, equivalently lower-) central series by $\operatorname{cl}(G)$.
We denote by $\operatorname{rk}(G)$ the \emph{rank} of $G$, defined by $\operatorname{rk}(G) := \dim_{\mathbb{R}} \mathfrak{g}/\mathfrak{g}_{[2]}$, where $\mathfrak{g}$ is the Lie algebra of $G\otimes \mathbb{R}$. That is, we have $G_1/G_2 \cong \mathbb{R}^{m_1}\times \mathbb{Z}^{m_1'}$ for some uniquely determined $m_1,m_1'\in \NN_0$ and  $\operatorname{rk}(G) = m_1+m_1'$.
\end{definition}

\subsection{Cohomology of Mal'cev groups} \label{sec:Cohom_nilpotent}
In this section we gather the results needed about the cohomological properties of Mal'cev groups, which turn out, not surprisingly, to be very much alike those for csc nilpotent Lie groups.

\begin{prop}\label{prop:malcev_cohom_fd}
Mal'cev groups are cohomologically finite dimensional.
\end{prop}

\begin{proof}
Let $G$ be a Mal'cev group and let $\tilde{G}$ be the csc, nilpotent Lie group in which $G$ is cocompact. First note that the cohomology of $G$ stops after degree $d:=\dim(\tilde{G})$; indeed, for a Fr{\'e}chet $G$-module $\E$, the Shapiro lemma \cite[III, Proposition 4.1]{Guichardet:Cohomologie} gives
\[
\Cohom^n(G,\E)\simeq \Cohom^n(\tilde{G}, \operatorname{Ind}(\E)), 
\]
and for $n> d$ the right hand side vanishes (for instance by the van Est theorem \cite[III, Corollaire 7.2]{Guichardet:Cohomologie}).
Let $V$ be a continuous, finite dimensional $G$-module; we prove the statement by induction on $d=\dim(\tilde{G})$. In the case $d=1$, we have $G\simeq \RR$ or $G\simeq \ZZ$ and both of these are cohomologically finite dimensional. 
For the inductive step, let $G$ be a Mal'cev group with $d$-dimensional ambient Lie group and take  a central subgroup $Z\leq G$, isomorphic to either $\ZZ$ or $\RR$, such that $G/Z$ is again a Mal'cev group whose Mal'cev completion has dimension $d-1$. Then, as we just saw, $Z$ is cohomologically finite dimensional and thus $\Cohom^q(Z, V)$ is finite dimensional and, in particular, Hausdorff \cite[III, Proposition 2.4]{Guichardet:Cohomologie}, so the Hochschild-Serre spectral sequence exists \cite[III,  n$^\text{o}$ 5.1]{Guichardet:Cohomologie} and has $E_2$-term 
\[
E^{pq}_2= \Cohom^p(G/Z, \Cohom^q(Z,V)).
\]
So $E^{pq}_2=0$ whenever $q>2$ or $p>d-1$ and all non-vanishing terms are finite dimensional by the induction hypothesis; thus also $\Cohom^n(G,V)$ is finite dimensional.
\end{proof}

We now recall some well-known results concerning the continuous cohomology of nilpotent groups \cite{Delorme:1-Cohomologie,Shalom:Harmonic,Bader-Rosendal-Sauer:Vanishing}. In order to have the standard operator theoretic tools at our disposal and to comply with the standing assumption in \cite{Guichardet:Cohomologie} that vector spaces are complex,  in  the rest of this section the terminology `$\mathcal{H}$ is a unitary Hilbert $G$-module' will mean that $\mathcal{H}$ is 
a complex Hilbert space with a continuous, unitary $G$-action.
This, however, is not a serious restriction since in all our applications we will be able to pass from the setting of real topological vector spaces to the complex ditto via the standard complexification procedure, as one has 
$\PolCohom{d}{n}{}{G}{\mathcal{E}\otimes_{\RR} \mathbb{C}} \cong \PolCohom{d}{n}{}{G}{\mathcal{E}} \otimes_{\RR} \mathbb{C}$. \\

We first recall Shalom's property $H_T$ introduced in \cite{Shalom:Harmonic}.  Here, as usual, $\Cohomred^n(G,\mathcal{H})$ denotes the reduced cohomology; i.e.~the maximal Hausdorff quotient of the ordinary cohomology.

\begin{definition}[Property $H_T$ {\cite{Shalom:Harmonic}}] \label{def:H_T}
A lcsc group $G$ is said to have \emph{property $H_T$} if for any continuous unitary  $G$-module $\mathcal{H}$ with $\mathcal{H}^G = 0$ one has
$\Cohomred^n(G,\mathcal{H}) = 0$
for any $n\in \NN$. 
\end{definition}

A well known result, essentially due to Delorme, concerning the vanishing of cohomology for nilpotent (Lie) groups, ensures that such groups have property $H_T$.
The classical form of the statement is the following: 

\begin{theorem}[{\cite[Theorem 10.1]{Blanc:Cohomologie}}]  \label{thm:Delorme}
Let $G$ be a csc nilpotent Lie group. For every irreducible, continuous unitary Hilbert $G$-module $\mathcal{H}$ such that $\mathcal{H}^G = 0$, we have $\Cohom^n(G,\mathcal{H}) = 0$ for all $n\in \mathbb{N}_0$. In particular $G$ has property $H_T$.
\end{theorem}
Note that the latter statement in the theorem above does indeed follow from the former, since any unitary representation is a direct integral of irreducible representations, and property $H_T$ therefore follows from \cite[Theorem 7.2]{Blanc:Cohomologie}.  It will be convenient for us to have the following alternate form of Theorem \ref{thm:Delorme}, which at the same time generalizes the statement to the class of Mal'cev groups.

\begin{theorem}\label{thm:Delorme_prime}
If $G$ is a Mal'cev group and $\H$ is a unitary $G$-module, then there exists an increasing sequence $(\H_i)_{i\in \NN}$ of closed, $G$-invariant subspaces of $\H$ with dense union, such that $\H^G\subset \H_i$ for each $i\in \NN$ and such that the inclusion map induces an isomorphism $\Cohom^n(G,\H^G)\simeq \Cohom^n(G,\H_i)$ for each $n\in \NN$. 
\end{theorem}
Note that it is part of the conclusion that $\Cohom^n(G,\mathcal{H}_i)$ is Hausdorff if $\mathcal{H}^G$ is finite dimensional, since, in this case, $\Cohom^n(G,\H^G)$ is finite dimensional (and thus Hausdorff  \cite[III, Proposition 2.4]{Guichardet:Cohomologie}) by Proposition \ref{prop:malcev_cohom_fd}. For the proof of  Theorem \ref{thm:Delorme_prime} the following lemma is convenient.

\begin{lem}\label{existence-of-subspace-seq}
Let $G$ be lcsc group, $\H$ be a unitary $G$-module and assume that there exists a central element $z\in G$ such that the corresponding unitary $u\in \BB(\H)$ satisfies that $T:=u-1$ is injective. Then there exists an increasing sequence $(\H_i)_{i\in \NN}$ of closed, $G$-invariant subspaces with dense union and such that $\Cohom^n(G,\H_i)=0$ for each $i\in \NN$ and $n\in \NN_0$.
\end{lem}

\begin{proof} Denote the representation by $\pi$.
Since $u:=\pi(z)$ is unitary, the operator $T:=u-\bbb$ is normal and hence  the Borel functional calculus may be applied to $T$. As $z$ is central, $T$ commutes with $\pi(G)$  and hence so do its spectral projections. Because $T$ is assumed to be injective, its spectral projections $P_n:=\chi_{\sigma(T)\setminus \{z\mid |z|\leq \sfrac{1}{n} \}}(T)$ are increasing and converging strongly to $\bbb$,  
and since each $P_n$ commutes with $\pi(G)$, its range $\H_n:=P_n(\H)$ is a closed $G$-invariant subspace; we denote the restricted representation of $G$ on $\H_n$ by $\pi_n$. Since $u$ is unitary, the injectivity of $T$ implies that its range is dense
and from this it follows that the operator $\pi_n(z)-1$, which is simply $T\restriction_{\H_n}$, has dense range as well. The operator  $\pi_n(z)-1$ is furthermore bounded away from zero,  and  thus invertible on $\H_n$, and by \cite[III, Proposition 3.1]{Guichardet:Cohomologie} this implies that $\Cohom^k(G,\H_n)$ vanishes for each $k\in \NN_0$ and $n\in \NN$.
\end{proof}

We are now ready to give the proof of Theorem \ref{thm:Delorme_prime}. In the proof we will several times use the fact that for a unitary $\ZZ$-module $\H$, having $\H^{\ZZ}=\{0\}$ is equivalent with $u_1-\bbb$ acting injectively (here $u_1$ is the unitary corresponding to $1\in \ZZ$), a fact that not true for unitary $\RR$-modules, 
which  accounts for the distinction between discrete and continuous one dimensional subgroups present in the proof.
\begin{proof}[{Proof of Theorem  \ref{thm:Delorme_prime}}]
By splitting $\H$ as $\H=\H^G \oplus \H^{G\perp}$, it suffices to treat the case where $\H^G=\{0\}$. Denote by $\tilde{G}$ the csc, nilpotent Lie group in which $G$ embeds cocompactly; we now prove the statement by induction on $d:=\dim(\tilde{G})$.
 For $d=1$, the group $G$ is isomorphic to either $\RR$ or $\ZZ$, and in the latter case the statement follows directly from Lemma \ref{existence-of-subspace-seq}. If $G\simeq \RR$, consider the subgroup $Z$ corresponding to $\ZZ\leq \RR$ and split $\H$ as $\H^Z\oplus \H^{Z\perp}$. Since $G$ is abelian, this is a splitting of $\H$ as a unitary $G$-module and by Lemma \ref{existence-of-subspace-seq} we now get increasing, closed, $G$-invariant subspaces $\K_i\leq \H^{Z\perp} $ with dense union an vanishing cohomology. Put $\H_i:= \H^{Z} \oplus  \K_i $. Then we have $\Cohom^n(G,\H_i)=\Cohom^n(G,\H^Z)$, so our task is to prove that the later vanishes in all degrees. To this end, note that
\[
\Cohom^q\left(Z, \H^Z\right)=
\begin{cases} \H^Z &\mbox{if } q=0,1  \\
\{0\}& \mbox{otherwise}  \end{cases}  
\]
and, in particular, $\Cohom^q(Z, \H^Z)$ is Hausdorff for all $q\in \NN_0$.  The Hochschild-Serre spectral sequence therefore exists \cite[III,  n$^\text{o}$ 5.1]{Guichardet:Cohomologie} and has $E_2$-term
\[
E_2^{pq}= \Cohom^p(G/Z, \Cohom^q(Z,\H^Z))=
\begin{cases} \Cohom^p(G/Z,\H^Z) &\mbox{if } q=0,1  \\
\{0\}& \mbox{otherwise}  \end{cases} 
\]
However, since $G/Z\simeq S^1$ is compact, we have $\Cohom^p(G/Z,\H^Z) =\{0\}$ when $p> 0$ \cite[III, Corollaire 2.1]{Guichardet:Cohomologie} and in degree zero we have
\[
\Cohom^0\left(G/Z,\H^Z\right)= \left(\H^Z\right)^{G/Z}=\H^G=\{0\}.
\]
Thus, $E_2^{pq}=\{0\}$ for all $p,q\geq 0$ and hence $\Cohom^n(G,\H^Z)=\{0\}$, as claimed.\\
For the inductive step, let $G$ be  a Mal'cev group with $\dim(\tilde{G})=d$ and choose central subgroups  $Z\leq Z' \leq G$ such that $Z \simeq \ZZ$,  $K:=Z'/Z$ is compact and $G/Z'$ is again a Mal'cev group whose ambient csc nilpotent Lie group has dimension $d-1$; this is always possible since $G_{\text{cl}(G)}\simeq \ZZ^k\oplus \RR^l$ for some $k,l\in \NN_0$ so we have that $K$ is either trivial or $S^1$. Now decompose $\H=\H^Z\oplus \H^{Z\perp}$ and note that the decomposition respects the $G$-action since $Z$ is central. For the restricted action $G\curvearrowright \H^{Z\perp}$ we have, by construction, a central element such that the corresponding unitary acts without fixed points,  so by Lemma \ref{existence-of-subspace-seq} we get a sequence of closed $G$-equivariant subspaces $\K_i\subset \H^{Z\perp}$ with dense union and such that $\Cohom^n(G,\K_i)=\{0\}$ for all $n\in \NN_0$ and $i\in \NN$. Next split $\H^Z=(\H^Z)^K \oplus (\H^Z)^{K\perp}$, and since $K$ is central in $G/Z$ this decomposition respects the natural $G/Z$-action. On $(\H^Z)^K$ we get an induced action of $(G/Z)/K=G/Z'$ without non-trivial fixed points, so the induction takes over and  provides us with an increasing family of closed $G/Z'$-invariant subspaces $\L_i\leq (\H^Z)^K $ for which $\Cohom^n(G/Z', \L_i)=\{0\}$ for all $i\in \NN$ and $n\in \NN_0$. We now define
\[
\H_i:=\L_i\oplus \left(\H^Z\right)^{K\perp} \oplus \K_i\leq \left(\H^Z\right)^K\oplus \left(\H^Z\right)^{K\perp} \oplus \H^{Z\perp}=\H,
\]
and note that, as $\Cohom^n(G,\K_i)=\{0\}$, we have to show that $\Cohom^n(G,\L_i\oplus (\H^Z)^{K\perp})=\{0\}$ for all $i\in \NN$ and $n\in \NN_0$. As in the case $d=1$, this can be deduced by a spectral sequence argument: Since $\L_i\oplus (\H^Z)^{K\perp}\leq \H^Z$ we have
\[
\Cohom^q\left(Z, \L_i\oplus (\H^Z)^{K\perp}\right)=
\begin{cases}\L_i\oplus (\H^Z)^{K\perp} &\mbox{if } q=0,1  \\
\{0\}& \mbox{otherwise}  \end{cases}  
\]
so the Hochschild-Serre spectral sequence exists \cite[III,  n$^\text{o}$ 5.1]{Guichardet:Cohomologie} and has $E_2$-term
\[
E_2^{pq}= \Cohom^p\left(G/Z, \Cohom^q\big(Z,\L_i\oplus (\H^Z)^{K\perp}\big)\right)=
\begin{cases} \Cohom^p\left(G/Z,\L_i\oplus (\H^Z)^{K\perp}\right) &\mbox{if } q=0,1  \\
\{0\}& \mbox{otherwise}  \end{cases} 
\]
Since $K\leq G/Z$ is central and compact another application of the Hochschild-Serre spectral sequence (similar to the one carried out above in the case $d=1$)  
yields that
\[
\Cohom^p\left(G/Z,\L_i\oplus (\H^Z)^{K\perp}\right)\simeq \Cohom^p\left((G/Z)/K,\left(\L_i\oplus (\H^Z)^{K\perp} \right)^K\right)=\Cohom^p\left(G/Z',\L_i \right)=\{0\}.
\]
Thus $E_2^{pq}=\{0\}$ for all $p,q\geq 0$ and we conclude that $\Cohom^n(G,\H_i)=\{0\}$ for all $n\in \NN_0$ and $i\in \NN$, as desired.
\end{proof}

\begin{cor}\label{cor:Malcev_prop_H_T}
Mal'cev groups have property $H_T$.
\end{cor}
\begin{proof}
Let $G$ be a Mal'cev group and let $\H$ be a unitary Hilbert $G$-module without fixed points, and choose, according to Theorem \ref{thm:Delorme_prime}, an increasing sequence $\H_i\leq \H$ of closed, $G$-invariant subspaces with vanishing cohomology and dense union. Denote by $P_i$ the orthogonal projection onto $\H_i$; then the sequence $(P_i)_{i \in \NN}$ converges strongly to $\bbb$, and hence the convergence also holds uniformly (in the Hilbert space norm) on compact subsets of $\H$.
Fix an $n\in \NN$ and a continuous $n$-cocycle $c\colon G^n\to \H$. Since $\H_i$ is $G$-invariant, the projection $P_i$ commutes with the $G$-action, so  the map $c_i:=P_ic(-)\colon G^n \to \H_i$ is again a cocycle and hence inner by the defining properties of $\H_i$. Viewing $c_i$ as a sequence of cocycles with values in $\H$, 
we are therefore done if we can show that $(c_i)_i$ converges to $c$ in the standard topology on $\operatorname{Z}^n(G,\H)$ given by uniform convergence on compact subsets. For a compact set $K\subset G^n$, by continuity of $c$ the subset $c(K)\subset \H$ is also compact and thus
\[
\sup_{g\in K}\| c(g)- c_i(g)\|=\sup_{\xi \in c(K)}\|(\bbb-P_i)\xi\|\to 0. \qedhere
\]
\end{proof}
Observe that, in the proof just given,  we only used the fact that Mal'cev groups satisfy the conclusion of Theorem \ref{thm:Delorme_prime} to conclude that they have property $H_T$, and for the sake of generality it is convenient to promote this property to a definition:
\begin{defi}
A lcsc group  $G$ is said to have \emph{strong property $H_T$} if for any continuous unitary $G$-module $\H$
there exists an increasing sequence $\H_i$ of closed, $G$-invariant subspaces with dense union,  such that each of them contains $\H^G$  and such that the inclusion induces an isomorphism $\Cohom^n(G,\H^G)\simeq \Cohom^n(G,\H_i)$ for all $n\in \NN_0$.

\end{defi}

The following corollary provides a very direct and useful extension of Theorem \ref{thm:Delorme}.

\begin{corollary} \label{cor:Delorme_prime}
Let $G$ be cohomologically finite dimensional lcsc group with strong property $H_T$. If $\H$ is a continuous, unitary Hilbert $G$-module 
with $\dim_{\mathbb{R}} \mathcal{H}^G < \infty$, and $\mathcal{F}$ is a continuous, finite dimensional $G$-module with $\mathcal{F} = \mathcal{F}^{G(d)}$ for some $d\in \mathbb{N}$,  then the natural inclusion map $\mathcal{H}^{G}\otimes \mathcal{F} \rightarrow \H\otimes \F$ induces an isomorphism
\begin{equation}
\xymatrix{ \Cohom^n(G,\mathcal{H}^G\otimes \mathcal{F}) \ar[r]^>>>>>{\cong} & \underline{\Cohom}^n(G,\H\otimes \F) }.\notag
\end{equation}
\end{corollary}

\begin{proof}
Indeed, denoting $(\mathcal{H}^G)^{\perp}$ by $\mathcal{K}$ we have the following decomposition (respecting the topology)
\begin{equation}
{\operatorname{{H}}}^n(G,\mathcal{H}\tens \mathcal{F}) = {\operatorname{{H}}}^n(G,\mathcal{K}\otimes \mathcal{F}) \oplus {\operatorname{{H}}}^n(G,\mathcal{H}^G\otimes \mathcal{F}).\notag
\end{equation}
and since $\dim(\H^G \tens \F)<\infty$,  ${\operatorname{{H}}}^n(G,\mathcal{H}^G\otimes \mathcal{F})$ is also finite dimensional and hence automatically reduced (\cite[III, Proposition 2.4]{Guichardet:Cohomologie}).
Thus, we have to show that ${\underline{\operatorname{{H}}}}^n(G,\mathcal{K}\otimes \mathcal{F})=0$. As $G$ has strong property $H_T$, 
we obtain an increasing sequence $(\K_i)$ of closed, $G$-invariant subspaces of $\mathcal{K}$ with dense union and vanishing cohomology, and we now prove, by induction on $d$, that $\Cohom^n(G,\mathcal{K}_i\otimes \mathcal{F}) = \{0\}$ for all $i$. Indeed, if $d=1$ then $\mathcal{F}=\mathcal{F}^G$, and the action therefore trivial, and since ${\operatorname{{H}}}^n(G,\mathcal{K}_i)=0$ and $\mathcal{K}_i\tens \mathcal{F}$ is, as a $G$-module, just a finite direct sum of copies of $\mathcal{K}_i$,  we also have ${\operatorname{{H}}}^n(G,\mathcal{K}_i \tens \mathcal{F})=\{0\}$. For the inductive step, consider the short exact sequence 
\[
0\to \mathcal{F}^{G(d-1)}\tens \mathcal{K}_i \to \mathcal{F}^{G(d)}\tens \mathcal{K}_i \to \mathcal{F}^{G(d)}/\mathcal{F}^{G(d-1)}\tens \mathcal{K}_i\to 0.
\]
The induction hypothesis  implies that ${\operatorname{{H}}}^n(G, \mathcal{F}^{G(d-1)}\tens \mathcal{K}_i )=\{0\}$, and the induced action on the quotient  $\mathcal{F}^{G(d)}/\mathcal{F}^{G(d-1)}$ is  easily seen to be trivial 
so, as in the case $d=1$, we also get ${\operatorname{{H}}}^n(G, \mathcal{F}^{G(d)}/\mathcal{F}^{G(d-1)}\tens \K_i)=\{0\}$. Since $\F$ is assumed to be equal to $\F^{G(d)}$,  the long exact sequence in cohomology now shows that also ${\operatorname{{H}}}^n(G,\mathcal{F}\tens\K_i)=\{0\}$.
Finally, since $\mathcal{K}\otimes \mathcal{F}$ admits a continuous $G$-equivariant projection $P_{\mathcal{K}_i}\otimes \bbb$ onto $\mathcal{K}_i\otimes \mathcal{F}$ for all $i$,
and we have that $P_{\mathcal{K}_i}\tens \bbb$ converges strongly to $\bbb\tens \bbb$, we conclude, as in the proof of Corollary \ref{cor:Malcev_prop_H_T}, that $\underline{{\operatorname{{H}}}}^n(G, \mathcal{K} \tens \mathcal{F})=0$.
\end{proof}

\section{Polynomials on Mal'cev groups} 

Building on the general results in the previous sections, we can now give a complete description of the polynomials on a Mal'cev group. Let therefore $G$ be a Mal'cev group and let $(g_{i,j})$ be a Mal'cev basis of $G$. Then for each pair $(i_0,j_0)\in \Bultiindex{B}{\multirank(G)}$ we consider the map $\zeta_{g_{i_0,j_0}} \colon G\rightarrow \mathbb{R}$ given by
\begin{equation} \label{eq:Malcev_pol_def}
\zeta_{g_{i_0,j_0}} \colon  \prod_{(i,j)\in \Bultiindex{B}{\multirank(G)}} g_{i,j}^{t_{i,j}} \longmapsto t_{i_0,j_0},
\end{equation}
where $t_{i,j}$ ranges over the set $Z_{i,j}$, either equal to $\ZZ$ or $\RR$ (see Section \ref{sec:Malcev} for this and  Section \ref{sec:notation} for the definition of the multiindex notation). 
More generally, we will need the following notation: for any multi-index $\multiindex{d}\in \Multiindex{D}{d}{\multidim(G)}$ we define
\begin{equation}
\zeta_{\multiindex{d}} := \prod_{(i,j)\in \Bultiindex{B}{\multirank(G)}} \zeta_{g_{i,j}}(-)^{d_{i,j}} \colon \prod_{(i,j)\in \Bultiindex{B}{\multirank(G)}} g_{i,j}^{t_{i,j}} \longmapsto \prod_{(i,j)\in \Bultiindex{B}{\multirank(G)}} t_{i,j}^{d_{ij}}.\notag
\end{equation}
With this notation at our disposal, we can now give the promised description of polynomials on Mal'cev groups.
\begin{theorem} \label{thm:pol_Malcev_description}
Let $G$ be a 
Mal'cev group and let $(g_{i,j})_{i,j}$ be a Mal'cev basis. Then for all $(i_0,j_0)\in  \Bultiindex{B}{\multirank(G)}$ the map $\zeta_{g_{i_0,j_0}}$ defined above is a polynomial map of degree $\deg \zeta_{g_{i_0,j_0}} = i_0$. Furthermore, the set $\{ \zeta_{\multiindex{d}} \mid \multiindex{d}\in \Multiindex{D}{d}{\multidim(G)} \}$ is a linear basis of $\operatorname{Pol}_d(G)$.
\end{theorem}

\begin{proof}

Assume first that $G$ is a csc nilpotent Lie group. We shall then show, by induction on $m:=\dim(G)$,  that $\deg(\zeta_{g_{ij}})=i$; the case of $m=1$ being trivial. 
For the inductive step, we need a bit of notation.  We first recall Leibman's definition of lc-polynomials 
from \cite{Leibman:Polynomial_mappings}: if $G$ is nilpotent with $\text{cl}(G)=c$ and $H$ is any group then $\varphi\colon H\to G$ is an \emph{lc-polynomial of lc-degree at most
$(1,2,\cdots, c)$}    if for all $i=1,\dots, c$ and $h_{1},\dots, h_{i+1}\in H$: $\rdiff_{h_1}\circ \cdots \cdot \circ \rdiff_{h_{i+1}}(\varphi) (H) \in G_{[i+1]}$,  where the differentiation operator is defined as $\rdiff_h(\varphi)(g):=\varphi(g)^{-1}\varphi(gh)$. In particular, since $G_{[c+1]}=\{\bbb\}$ this forces $\varphi$ to be a ``\hspace{0.01cm}polynomial of degree at most $c$\hspace{0.04cm}''; i.e.to satisfy $\rdiff_{h_1}\circ \cdots \circ \rdiff_{h_{c+1}}\xi=\bbb_G$.
The main virtue of the class of lc-polynomials is that they, by \cite[Proposition 3.4]{Leibman:Polynomial_mappings},  form a group under pointwise multiplication.
Now, if $(g_{ij})$ is  a Mal'cev basis and $\xi\colon G\to \RR$ is a polynomial of degree at most $i$, then the map $\varphi\colon G\to G$ given by $\varphi(h)=g_{ij}^{\xi(h)}$ is an lc-polynomial with lc-degree  at most $(1,\dots, c)$. To see this, note that
\begin{align}\label{diff-of-phi}
\rdiff_{g_1}\circ \cdots \circ \rdiff_{g_{l+1}} (\varphi)(h)=g_{ij}^{\rdiff_{g_1}\circ \cdots \circ \rdiff_{g_{l+1}}(\xi)(h)}
\end{align}
and since $g_{ij}^{\RR}\leq G_{[i]}$ we have $\rdiff_{g_1}\circ \cdots \circ \rdiff_{g_{l+1}} (\varphi)(h)\in G_{[i]} \leq G_{[l+1]}$
when $l+1\leq  i$   and $\rdiff_{g_1}\circ \cdots \circ \rdiff_{g_{l+1}} (\varphi)(h)=\bbb \in G_{[l+1]}$ when $l+1>i\geq \deg(\xi)$. 
We are now ready to return to  the inductive step. Fix some $z:=g_{c,j_{0}}\in G_{[c]}\leq Z(G)$ and denote $G/z^{\RR}$ by $\bar{G}$ and  the quotient map $G\to \bar{G}$ by $\pi$.  Note that the $g_{ij}$'s with $(i,j)\neq (c,j_0)$ project onto a Mal'cev basis $\bar{g}_{ij}$ for $\bar{G}$ and hence $\zeta_{\bar{g}_{ij}}$ has degree $i$ by the induction hypothesis. But $\zeta_{g_{ij}}=\zeta_{\bar{g}_{ij}}\circ \pi$ so we also obtain $\deg(\zeta_{g_{ij}})=i$. Thus, we only have to prove that $\deg(\zeta_z)=c$. To this end, write
\begin{align}\label{eq:pol_Malcev_description_zeta_z}
z^{\zeta_z(h)} = h\cdot \left( \prod_{(i,j)\neq \operatorname{cl}(G),j_0} g_{i,j}^{\zeta_{g_{i,j}}(h)} \right)^{-1}.
\end{align}
As already mentioned, by  \cite[Proposition 3.4]{Leibman:Polynomial_mappings}  the set of lc-polynomials from $G$ to $G$ of degree at most (1,\dots, c) is a group under pointwise multiplication, so since the identity map is clearly such an lc-polynomial, if we can prove that each of the factors  $h\mapsto g_{i,j}^{\zeta_{g_{i,j}}(h)}$ in the product has lc-degree at most $(1,\dots, c)$ we obtain that also $\varphi\colon h\mapsto z^{\zeta_z(h)}$ has lc-degree at most $(1,\dots, c)$.
However,  as we just saw, when $(i,j)\neq (c,j_0)$, $\zeta_{g_{ij}}$ has degree $i$ and hence $h\mapsto g_{ij}^{\zeta_{g_{ij}}(h)}$ has lc-degree at most $(1,\dots, c)$ as desired , and from this it follows, using \eqref{diff-of-phi}, that $ \zeta_z\colon G\to \RR$ has degree at most $c$. We still need to prove that $\deg(\zeta_z)=c$, but if  $\deg(\zeta_z)\leq c-1$  then, by Corollary \ref{cor:leibman-cor}, the (unital) polynomial $\zeta_z$ must vanish on $G_{[c]}$,  which cannot be the case as $z\in G_{[c]}$ and $\zeta_z(z)=1$.\\

For the second part of the statement, denote again by $c=\text{cl}(G)$ the class  of $G$ and define for $d\geq 0$
\[
\A_d:=\spann_{\RR}\left\{\zeta_{\mathbf{d}}\mid \mathbf{d}\in \Multiindex{D}{d}{\multidim(G)}\right\}
\]
(recall that $\mathbf{d}=(d_{ij})\in \Multiindex{D}{d}{\multidim(G)}$ iff $\sum_{ij}id_{ij}\leq d$) and put $\A_{-\infty}=\{0\}$.  We need to prove that $\Pol_d(G)=\A_d$. Here the inclusion ``$\supseteq$'' follows from what was already proven and the general estimate $\deg(\xi\eta)\leq \deg(\xi)+\deg(\eta)$. To see the opposite inclusion,
we run an induction on $m=\dim(G)$, in which the  base case $m=1$ is trivial. For the inductive step, let $\xi \in \Pol_d(G)$  be given and fix $z:=g_{c,j_0}\in G_{[c]}\leq Z(G)$; we now run a finite subinduction on the minimal number $n\in \NN$ such that $\rdiff_z^{(n)}(\xi)=0$. If $n=1$, then  $\rdiff_z\xi=0$ 
and hence, by Lemma \ref{lem:poly-on-quotient}, induces a polynomial $\bar{\xi}$ on $\bar{G}:=G/z^{\RR}$ which, by the primary induction,  can be written as a linear combination of products of 
$(\zeta_{\bar{g}_{i,j}})_{(i,j)\neq (c,j_0)}$.  This means that $\xi\in \A_d$  as it can be written as a linear combination of the $\zeta_{\mathbf{d}}$'s even without using $\zeta_z$. For $n=2$, a direct computation shows that
\[
\rdiff_z\left( \rdiff_z(\xi)\zeta_z - \xi \right)=0.
\]
Thus, by the $n=1$ case just covered, this means that $\rdiff_z(\xi)\zeta_z - \xi \in \A_d$. Moreover, by Lemma \ref{lma:pol_diff_degree_wrt_lcs},  $\deg(\rdiff_z(\xi))\leq d\minusdot c $ and since $\rdiff_{z}(\rdiff_z(\xi))=0$, the $n=1$ case gives that $\rdiff_z\xi\in \A_{d\minusdot c}$. Hence $\rdiff_z(\xi)\zeta_z\in \A_d$ and thus also $\xi\in \A_d$. The general case  is a bit more involved, but overall builds on the same idea  used for $n=2$, and for that we need some more detailed information about differentiation and integration with respect to $z$, contained in the following three claims.
\setcounter{claim-counter}{0}
\begin{claim}
We have $\rdiff_z(\zeta_z^k) \in \spann_{\NN} \{ \zeta_z^l\mid 0\leq l\leq k-1 \}$.
\end{claim}
\noindent We stress the fact that the span appearing in Claim 1 is over the naturals, so that, in particular, the leading coefficient is non-zero.

\begin{proof}[Proof of Claim 1]
For $k=1$, $\rdiff_z(\zeta_z)=1$ and for $k=2$ we have $\rdiff_z(\zeta_z^2)=1+2\zeta_z$. The general case follows inductively: assuming the result for $k-1$ we have, using the Leibniz rule \eqref{eq:leibniz_rule}, that
\[
\rdiff_z(\zeta_z^k)=\rdiff_z(\zeta_z^{k-1}\zeta_z)= \rdiff_z(\zeta_z^{k-1})\cdot1 + \rdiff_z(\zeta_z^{k-1})\cdot \zeta_z + \zeta_z^{k-1}\cdot 1
\]
By the induction hypothesis, $\rdiff_z(\zeta_z^{k-1})\in \spann_{\NN} \{ \zeta_z^l\mid 0\leq l\leq k-2 \}$ and thus  $\rdiff_z(\zeta_z^{k-1})\cdot \zeta_z\in \spann_{\NN} \{ \zeta_z^l\mid 0\leq l\leq k-1 \}$ and hence also $\rdiff_z(\zeta_z^k)\in \spann_{\NN} \{ \zeta_z^l\mid 0\leq l\leq k-1 \}$.
\end{proof}
\begin{claim}
For each $k\in \NN_0$ there exists  $\Upsilon_k\in \spann_{\RR} \{ \zeta_z^l\mid 0\leq l\leq k+1 \} $ such that $\rdiff_z\Upsilon_k=\zeta_z^k$. 
\end{claim}
\begin{proof}[Proof of Claim 2]
For $k=0$ this is clear as $\rdiff_z\zeta_z=1$ and for $k=1$ we have $\rdiff_z(\zeta_z^2)=1+2\zeta_z$ so $\Upsilon_1:=\frac{1}{2}\zeta_z^2 - \frac{1}{2}\zeta_z$ does the job.  The general case follows inductively: assume Claim 2 true for $k-1$. By Claim 1, we get $a_0,\dots, a_k\in \NN$ such that
\[
\rdiff_z(\zeta_z^{k+1})=a_0 + a_1\zeta_z +\cdots + a_{k}\zeta_z^k,
\]
and since $a_k\in \NN$ we have
\begin{align*}
\zeta_z^k&=\frac{1}{a_k}\left(\rdiff_z(\zeta_z^{k+1})- \sum_{i=0}^{k-1}a_{i}\zeta_z^{i} \right)
= \frac{1}{a_k}\left(\rdiff_z(\zeta_z^{k+1})- \sum_{i=0}^{k-1}a_{i}\rdiff_z\Upsilon_i \right)\\
&= \rdiff_z \underbrace{\left(\frac{1}{a_k}\zeta_z^{k+1}- \frac{1}{a_k}\sum_{i=0}^{k-1}a_{i}\Upsilon_i \right)}_{=:\Upsilon_k}.
\end{align*}
\end{proof}

\begin{claim}
For every $\xi \in \A_{d\minusdot c}$ there exists $\Xi\in \A_{d}$ such that $\rdiff_z \Xi=\xi$.
\end{claim}
\begin{proof}[Proof of Claim 3] For $\xi \in \A_{d\minusdot c}$, by grouping summands together according to their power of $\zeta_z$ we can write it as
\[
\xi=\sum_{k=0}^m r_k \eta_k\zeta_z^k
\]
(for some $m\in \NN$ and $r_k\in \RR$) where $\eta_k\in \A_{d\minusdot c\minusdot kc}$  and $\rdiff_z(\eta_k)=0$.  Putting $\Xi:= \sum_{k=0}^m r_k \eta_k\Upsilon_k$, where the $\Upsilon_k$'s are as in Claim 2, we get, using the Leibniz rule \eqref{eq:leibniz_rule}, that $\rdiff_z(\Xi)=\xi$. Moreover, by Claim 2 we have $\Upsilon_k\in \spann_{\RR}\{\zeta_z^l \mid 0\leq l\leq k+1\} $ and since $\eta_k\in \A_{d\minusdot c\minusdot kc}$
we conclude that $\Xi\in \A_d$ as desired.
\end{proof}
We can now finish the (sub-)induction argument, which is running over the minimal $n$ such that $\rdiff^{(n)}_z(\xi)=0$. Given  $\xi \in \Pol_{d}(G)$, $\rdiff_z\xi$ falls under the induction hypothesis and has degree at most $d\minusdot c$ by Lemma \ref{lma:pol_diff_degree_wrt_lcs}. Hence $\rdiff_z(\xi)\in \A_{d\minusdot c}$, so by Claim 3 there exists $\Xi\in \A_d$ such that $\rdiff_z\Xi=\rdiff_z(\xi)$. Thus,
$
\rdiff_z(\Xi-\xi)=0,
$
and by the base case ($n=1$) this means that $\Xi-\xi\in \A_d$. By construction, $\Xi\in \A_d$ and hence also $\xi\in \A_d$.   This finishes the proof that that $\Pol_d(G)=\A_d$,  and the linear independence of the polynomials $\{\zeta_{\multiindex{d}} \mid  \multiindex{d}\in \Multiindex{D}{d}{\multidim(G)}  \}$ is clear, since they pull back to linearly independent polynomials on  $\mathbb{R}^{\dim \mathfrak{g}}$ via the Mal'cev coordinates. This completes the proof in the case where $G$ is a csc nilpotent Lie group. \\

In the general case, we know that $G$ is cocompactly embedded in its Mal'cev completion $L$, and that we may choose a Mal'cev basis for $L$ based in $G$. Denote the dimension of $L$ by $n$. Then the Mal'cev coordinates gives a diffeomorphism
$ L \simeq \RR^n$ which identifies $G$ with a (cocompact) subset of the form $\ZZ^{m}\times \RR^{m'}$ where $m+m'=n$. Moreover, by what was just proven we know that polynomials on $L$ pull back to polynomials on $\RR^n$ via the Mal'cev coordinates, and polynomials on $\RR^n$ are uniquely determined on the subset $\ZZ^{m}\times \RR^{m'}$. Thus, polynomials on $L$ are uniquely determined by their values on $G$, so the restriction map $\text{res}_d\colon \Pol_d(L) \to \Pol_d(G)$ is injective for all $d\in \NN_0$. 
We now need to prove that it is also surjective. We first note that this is trivially the case when $d=0$, and we now proceed by induction on $d$. Assuming this to be true up to $d-1$ we have\footnote{The tensor product with $\CC$ is included in order to formally conform with framework in Section \ref{sec:Malcev}; see remarks preceding Definition \ref{def:H_T}.}
\begin{align*}
 \CC\tens_{\RR}\left(\Pol_d(G)/\Pol_{d-1}(G)\right) &\simeq   \CC\tens_{\RR}\Cohom^1(G,\Pol_{d-1}(G)) \tag{Prop.~\ref{prop:PolCohom_linear_description}}\\
 &\simeq  \Cohom^1(G,\CC\tens_{\RR}\Pol_{d-1}(G))\\
 &\simeq \Cohom^1(G, \CC\tens_{\RR} \Pol_{d-1}(L)) \\
 &\simeq \Cohom^1\left(L, \text{Ind}_G^L\left(\CC\tens_{\RR}  \Pol_{d-1}(L)\right)\right) \tag{\cite[III, Prop.~4.6]{Guichardet:Cohomologie}}\\
 &\simeq \Cohom^1\left(L, \operatorname{L}^2(L/G)\tens_{\CC}\CC\tens_{\RR}  \Pol_{d-1}(L)\right) \tag{\cite[Cor. E.2.6 (i)]{BHV}}\\
 &\simeq \Cohom^1(L, \CC\tens_{\RR}\Pol_{d-1}(L)) \tag{Cor.~\ref{cor:Delorme_prime}}\\
 &\simeq \CC\tens_{\RR}\left(\Pol_d(L)/\Pol_{d-1}(L)\right). \tag{Prop.~\ref{prop:PolCohom_linear_description}}
\end{align*}
Note that we may indeed apply Corollary \ref{cor:Delorme_prime} to obtain the penultimate equality, because $L^2(L/G)^L=\CC.1_{L/G}$ and $ \Cohom^1(G,\Pol_{d-1}(G))$, and hence also $ \Cohom^1(L, \operatorname{L}^2(L/G)\tens_{\CC}\CC\tens_{\RR}  \Pol_{d-1}(L))$, is finite dimensional (by Proposition \ref{prop:malcev_cohom_fd} and Corollary \ref{cor:cohom_fin_dim_implies_pol_fin_dim}) which implies that the latter is  automatically Hausdorff \cite[III, Proposition 2.4]{Guichardet:Cohomologie}. From this we conclude that $\Pol_d(G)$ and $\Pol_d(L)$ have the same (finite) linear dimension and hence the restriction map $\text{res}_d\colon \Pol_d(L) \to \Pol_d(G)$ must be surjective as well.  The only thing left to prove is that $\deg(\zeta_{g_{ij}})=i$ when $\zeta_{g_{ij}}$ is considered as a polynomial on $G$.  
However, as we saw above the restriction map $\Pol_d(L)\to \Pol_d(G)$ is a linear isomorphism for each $d\in \ZZ_*$ and this forces $\deg \circ  \text{ res}_d(\xi) =\deg(\xi)$ and we proved above that when considered a polynomial on $L$ the degree of $\zeta_{g_{ij}}$  is indeed $i$.
\end{proof}

\begin{rem}\label{rem:exp-description}
Theorem \ref{thm:pol_Malcev_description} describes the polynomials on a csc nilpotent Lie group in terms of a  Mal'cev basis, but by \cite[Proposition 1.2.7]{Corwin-Greenleaf} these may  equivalently be described  as those maps that pull back to classical polynomials on the associated Lie algebra via the exponential map.
\end{rem}

\begin{cor}\label{cor:degree-of-product}
For a Mal'cev group $G$ and $\xi,\eta\in \Pol(G)$ we have $\deg(\xi\cdot\eta)=\deg(\xi) \plusdot \deg(\eta)$.  In particular, $\deg(\zeta_{\multiindex{d}})=\sum_{i,j}id_{ij}$.
\end{cor}
\begin{proof}
As we saw above (cf.~\eqref{eq:product_rdiff} and the remarks preceding it), the inequality `$\leq$' is true for any group $G$ so we only need to prove the opposite. 
Upon picking a Mal'cev basis for $G$, by Theorem \ref{thm:pol_Malcev_description} we therefore have $\deg(\zeta_{\multiindex{d}})\leq \sum_{i,j}id_{ij}=:d$ for any multi-index $\multiindex{d}$. However, if $\deg(\zeta_{\multiindex{d}})<d$ then  $\zeta_{\multiindex{d}}\in \A_{d-1}$  by Theorem \ref{thm:pol_Malcev_description} which contradicts the linear independence of the basis also provided by Theorem \ref{thm:pol_Malcev_description}. \\
For the general claim about products, put $d:=\deg(\xi)$ and $d':=\deg(\eta)$  and note that the statement is trivial if either number is $-\infty$, so we may assume that this is not the case. Write, according to Theorem \ref{thm:pol_Malcev_description}, the polynomials as $\xi=\sum_{\multiindex{d}\in \Multiindex{D}{d}{\multidim(G)}} r_{\multiindex{d}}\zeta_{\multiindex{d}}$ and $\eta=\sum_{\multiindex{c}\in \Multiindex{D}{d'}{\multidim(G)}} s_{\multiindex{c}}\zeta_{\multiindex{c}}$. 
Due to the linear independence of the $\zeta_{\multiindex{d}}$'s, the only way that  we can have $\deg(\xi\cdot\eta)<d+d'$ is if
\[
\Big(\sum_{\multiindex{d}\in \Multiindex{D}{d}{\multidim(G)}^{=}} r_{\multiindex{d}}\zeta_{\multiindex{d}} \Big) \Big(\sum_{\multiindex{c}\in \Multiindex{D}{d'}{\multidim(G)}^{=}}s_{\multiindex{c}}\zeta_{\multiindex{c}} \Big)=
\sum_{\multiindex{e}\in \Multiindex{D}{d+d'}{\multidim(G)}}\Big( \sum_{ \substack{ \multiindex{d}\in \Multiindex{D}{d}{\multidim(G)}^{=}  \\   \multiindex{c}\in \Multiindex{D}{d'}{\multidim(G)}^{=} \\ \multiindex{d} +\multiindex{c}=\multiindex{e}  }} r_{\multiindex{d}}s_{\multiindex{c}} \Big)\zeta_{\multiindex{e}}
\]
has degree less than $d+d'$, and by what was already shown $\deg(\zeta_{\multiindex{c} + \multiindex{d}})=d+d'$ for all $\multiindex{d}\in \Multiindex{D}{d}{\multidim(G)}^{=}$ and $\multiindex{c}\in \Multiindex{D}{d'}{\multidim(G)}^{=}$. This therefore forces  the product on the left hand side to be zero, and pulling the polynomials back to $\RR^{\dim(\mathfrak{g})}$ via the Mal'cev coordinates we obtain classical polynomials in $\dim(G)$ variables, and since these constitute a domain one of the two factors needs to be zero, thus contradicting the the fact that $\deg(\xi)=d$ and $\deg(\eta)=d'$.
\end{proof}

\color{black}

\begin{cor} \label{lma:pol_diff_degree}
Let $G$ be a 
Mal'cev group and let $\xi\in \operatorname{Pol}(G)$. Then for every $g\in G$ we have $\deg \ldiff_g \xi \leq \deg \xi \minusdot \deg g$ and analogously for $\rdiff_g \xi$, where $\deg(g)$ is the degree with respect to the Mal'cev central series.
\end{cor}

\begin{proof}
The Mal'cev central series is  given as $G_i=G\cap (G\tens \RR)_{[i]}$ and hence $\deg(g)$ is the same whether we compute it with respect to the Mal'cev central series in $G$ or the lower central series in $G\tens \RR$. Furthermore, by  Theorem \ref{thm:pol_Malcev_description} we know that the restriction map $\Pol_d(G\tens \RR) \to \Pol_d(G)$ is bijective and degree preserving and the result therefore follows from Corollary \ref{lma:pol_diff_degree_wrt_lcs}.
\end{proof}

We will need also the following uniqueness results for polynomials.

\begin{lemma} \label{lma:pol_determined_subgroup}
Let $G$ be a 
Mal'cev group and let $(g_{i,j})_{i,j}$ be a Mal'cev basis for $G$. Denote $n := \operatorname{rk} G$ and let $G_0$ be the (not necessarily closed) subgroup of $G$ generated (algebraically) by $S:= \{g_{1,1},\dots ,g_{1,n}\}$. Finally, denote by $S^{\leq d}$ the set of words  on $S$ of length at most $d\in \mathbb{N}$. Then for  $d\in \mathbb{N}$ and  $\xi,\eta \in \operatorname{Pol}_d(G)$ we have
\begin{equation}
\xi = \eta \Leftrightarrow \xi_{\vert S^{\leq d}} = \eta_{\vert S^{\leq d}}.\notag
\end{equation}
In particular, any polynomial map  on $G$ is uniquely  determined by its values on $G_0$.
\end{lemma}

\begin{proof}
By Theorem \ref{thm:pol_Malcev_description}, the polynomials on $G$ and on its Mal'cev completion are the same, so by passing to the Mal'cev completion we may assume that $G$ is a csc, nilpotent Lie group.  As in \cite[Proposition 1.15]{Leibman:Polynomial_mappings}, we see that if $\xi_{\vert S^{\leq d}} = \eta_{\vert S^{\leq d}}$ then $\xi_{\vert G_0} = \eta_{\vert G_0}$, so the lemma will follow if we show that any polynomial vanishes on $G$ if it vanishes on $G_0$. Denote the closure of $G_0$ by $H$ and note that the quotient map $\pi\colon G\to G/G_{[2]}$ maps $H$ onto a cocompact subgroup (e.g. since $\pi(g_{1,1}),\dots, \pi({g_{1,n}})$ is a Mal'cev basis for the (abelian) quotient $G/G_{[2]}$ and all products of the form $\pi(g_{1,1})^{m_1}\cdots \pi({g_{1,n}})^{m_n}$ with $m_i\in \ZZ$ are contained in $\pi(H)).$
By \cite[Theorem 5.4.13]{Corwin-Greenleaf} (and the generalizing remarks following it in section 5.5), this implies that $H$ is cocompact in $G$.  In other words, $H$ is a Mal'cev group with Mal'cev completion $G$ and we therefore know that polynomials are uniquely determined by their values on $H$ and, by continuity, on its dense subgroup $G_0$.
\end{proof}

\begin{lemma} \label{lma:pol_determined_diff}
Let $G$ be a 
 Mal'cev group, $(g_{i,j})_{i,j}$ a Mal'cev basis and let $\xi,\eta\in \operatorname{Pol}(G)$. If $\xi(\bbb) = \eta(\bbb)$ and $\rdiff_{g_{1,j}} \xi = \rdiff_{g_{1,j}} \eta$ for all $j=1,\dots ,\operatorname{rk}(G)$ then $\xi=\eta$.
\end{lemma}

The statement may be deduced from the previous lemma, by showing that $\xi(g)=\eta(g)$ for all $g\in G_0$, by induction on word-length. Here is an alternative argument:

\begin{proof} For any $f\in C(G,\mathbb{R})$, the function $g\mapsto \rdiff_g f$ satisfies the $1$-cocycle identity, when $C(G,\RR)$ is considered a $G$-module for the right regular action. Hence we conclude that $\rdiff_g (\xi-\eta) = 0$ for all $g\in G_0$, the subgroup of $G$ generated by  $(g_{1,j})_j$. 
By Proposition \ref{prop:pol_diff_is_pol},  the map $g\mapsto \rdiff_g (\xi-\eta)(\bbb)$ is itself a  polynomial map on $G$, and since it vanishes on $G_0$ it vanishes on all of $G$ by the previous lemma. Thus, $\xi(g)=\eta(g)$ as desired.

\end{proof}


\subsection{The Hopf  algebra of polynomial maps} \label{sec:algebra_pol}
The space of polynomial maps $\operatorname{Pol}_d(G)$ may be seen as containing certain ``$d$'th order dual structure''. For instance, $\operatorname{Pol}_1(G)$, being essentially (that is, up to addition of some constant) the space of continuous group homomorphisms into $\mathbb{R}$, contains very precise information about the (torsion-free part of the) abelianization of $G$. 
In this section we elaborate on these considerations and Theorem \ref{thm:hom_existence} below makes precise in which way $\Pol(G)$ should be considered a dual object. 
\begin{rem}\label{rm:alg-geom-rem}
If $G$ is a csc nilpotent Lie group, when thinking of $G$ as the set of real points on an algebraic group, it follows from Theorem \ref{thm:pol_Malcev_description} that $\Pol(G)$ is the set of \emph{regular functions} on $G$, in the sense of algebraic geometry \cite{Borel-GTM}.
Many of the results deduced in this section therefore also follow from well-known results in algebraic geometry (e.g.~the fact that $\Pol(G)$ is a Hopf algebra \cite[Chapter 1]{Borel-GTM}), but for the sake of completeness, and since we wish to keep track of the degree of polynomials, which is not covered by algebraic geometry, we include the details below. 
\end{rem}

\begin{lemma}\label{lem: pol-of-product-grps}
Let $G$ and $H$ be  Mal'cev groups. The map $\alpha\colon C(G) \tens C(H) \to C(G\times H)$ given by $\alpha(\xi\tens \eta)(g,h):=\xi(g)\eta(h)$ restricts to an algebra-isomorphism $\Pol(G)\tens \Pol(H) \simeq \Pol(G\times H)$ which respects the grading given by the polynomial degree;  that is $\deg \alpha(\xi\otimes \eta) = \deg \xi + \deg \eta$.
\end{lemma}
Here, and in what follows, the symbol ``$\otimes$'' is used for the algebraic tensor product of real vector spaces.

\begin{proof}
If $(g_{ij})$ is a Mal'cev basis for $G$ and $(h_{kl})$ is one for $H$, then the set $((g_{ij}, \bbb), (\bbb, h_{kl}))_{i,j,k,l}$ is a Mal'cev basis for $G\times H$ and a direct computation verifies that $\alpha(\zeta_{g_{ij}}\tens 1)=\zeta_{(g_{ij}, \bbb)}$ and $\alpha(1\tens \zeta_{h_{kl}})=\zeta_{(\bbb,h_{kl})}$. From this it follows that $\alpha$, which is easily seen to be an algebra homomorphism, maps 
$\Pol(G) \tens \Pol(H)$  to $\Pol(G\times H)$ and as $\Pol(G\times H)$ is generated, as an algebra, 
by $(\zeta_{(g_{ij},\bbb)}, \zeta_{(\bbb,h_{kl})})_{ijkl}$ (Theorem \ref{thm:pol_Malcev_description}) the restriction of $\alpha$ is surjective.
Furthermore, by Theorem \ref{thm:pol_Malcev_description} the elements $\zeta_{\multiindex{d}}\tens \zeta_{\multiindex{c}}$ with $\multiindex{d}\in \Multiindex{D}{d}{\multidim(G)}$ and $\multiindex{c}\in \Multiindex{D}{d'}{\multidim(H)}$ constitute a basis for $\Pol_{d}(G)\tens \Pol_{d'}(H)$, and since $\alpha(\zeta_{\multiindex{d}}\tens \zeta_{\multiindex{c}})\in \{\zeta_{\multiindex{b}} \mid \Multiindex{D}{d+d'}{\multidim(G\times H)} \}$ 
it follows that $\alpha$ is injective on $\Pol_{d}(G)\tens \Pol_{d'}(G)$ for any $d,d'\in \NN$, and hence globally. That $\alpha $ is degree preserving can be seen by the same argument used to prove Corollary \ref{cor:degree-of-product}.
\end{proof}

The previous lemma, in particular, shows that, given Mal'cev groups $G$ and $H$, any linear map $\Psi \colon \operatorname{Pol}(G) \rightarrow \operatorname{Pol}(H)$ satisfying that $\deg \Psi(\xi) \leq \deg \xi$ for all $\xi\in \operatorname{Pol}(G)$, induces a map $\Psi\otimes \Psi \colon \operatorname{Pol}(G\times G) \rightarrow \operatorname{Pol}(H\times H)$ given by $\Psi(\xi\otimes \eta) = \Psi(\xi)\otimes \Psi(\eta)$  such that $\deg (\Psi\otimes \Psi)(\zeta) \leq \deg \zeta$ for all $\zeta \in \Pol(G\times G)$.

\begin{definition}[degree-preserving maps]
Let $G$ and $H$ be lcsc groups. We will say that a linear map $\Psi\colon \operatorname{Pol}(G) \rightarrow \operatorname{Pol}(H)$ is \emph{degree-preserving} if $\deg \Psi(\xi) \leq \deg \xi$ for all $\xi \in \operatorname{Pol}(G)$, and \emph{properly degree-preserving} if equality holds.
\end{definition}

\begin{definition}[strongly unital maps]
We say that a linear map $\Psi\colon \operatorname{Pol}(G) \rightarrow \operatorname{Pol}(H)$ is \emph{strongly unital} if it is unital and if $\Psi(\xi)(\bbb) = \xi(\bbb)$ for all $\xi\in \Pol(G)$.
\end{definition}

\begin{prop}\label{prop:mult-preserves-poly}
Let $G$ be a Mal'cev group with multiplication $m\colon G\times G\to G$ and let $\xi\in \Pol(G)$. Then $m^*(\xi):=\xi\circ m\in \Pol(G\times G)$ and $\deg(m^*(\xi))=\deg(\xi)$. That is, $m^*\colon \Pol(G) \to \Pol(G\times G)$ is properly degree preserving.

\end{prop}

\begin{proof}
That $m^*\xi$ is a polynomial for every polynomial  map $\xi$ follows from \cite{Leibman:Polynomial_mappings}: Indeed, we claim that multiplication $m\colon G\times G \rightarrow G$ is a polynomial map of lc-degree (cf.~\cite[Section 3]{Leibman:Polynomial_mappings}) $\operatorname{lc-deg} m = (1,\dots ,\operatorname{cl}(G))$. To see this, let $\pi_i\colon G\times G \rightarrow G, i=1,2$, denote the projections on the first and second factor,  respectively. Then $m(g) = \pi_1(g)\cdot \pi_2(g)$ is a pointwise product of homomorphisms, so the claim follows by \cite[Theorem 3.2]{Leibman:Polynomial_mappings}.  Next we now show that $m^*$ is properly degree preserving.
To this end, we first show that if $(\xi_i)_{i=1}^\infty$ is basis for $\Pol(G)$, chosen such that $\xi_0=1$, $\xi_i(\bbb)=0$ for $i\geq 1$ and  $\{\xi_i \mid \deg(\xi_i)\leq d\}$ is a basis for $\Pol_d(G)$ for every $d\in \NN_0$, and $m^*(\xi)$ is written as 
\[
m^*(\xi)=\sum_{i=0}^m \xi_i\tens \eta_i,
\]
with $\eta_i\in \Pol(G)$  then $\deg(\xi_i)\dotplus\deg(\eta_i)\leq \deg(\xi)$, from which we obtain $\deg(m^*(\xi))\leq \deg(\xi)$ by Lemma \ref{lem: pol-of-product-grps}.
When $\deg(\xi)=0$ this is basic linear algebra, 
and the general case now follows by induction on $n:=\deg(\xi)$: a direct computation shows that
\begin{align}\label{m-and-diff}
m^*(\rdiff_g(\xi))=\sum_{i=0}^m \xi_i \tens \rdiff_g\eta_i 
\end{align}
so the induction hypothesis gives that
\begin{align}\label{eq:deg-og-diff}
\deg(\xi_i) \dotplus \deg(\rdiff_g\eta_i)\leq \deg(\rdiff_g\xi)\leq \deg(\xi)-1=n-1, \ \text{ for all } g\in G.
\end{align}
For each non-constant $\eta_i$ there exists a $g\in G$ such that $\deg(\rdiff_g\eta_i)=\deg(\eta_i)-1\geq 0$ and hence
\[
\deg(\xi_i) \plusdot \deg(\rdiff_g\eta_i)= \deg(\xi_i) + \deg(\eta_i)-1\leq n-1.
\]
Thus, for those $i$ we have $\deg(\xi_i) \plusdot \deg(\eta_i)\leq n$ and $\deg(\xi_i)\leq n-1$. We may write
\[
\xi(g)=m^*(\xi)(g,\bbb) =\sum_{i: \deg(\eta_i)>0} \xi_i(g)\eta_i(\bbb) +\sum_{i:\deg(\eta)\leq 0} \xi_i(g) \eta_i(\bbb),
\]
and  since $\xi$ can be uniquely expressed as a linear combination of the elements $\{\xi_i \mid \deg(\xi_i)\leq n\}$, if $\deg(\xi_i)> n$ for some $i$, then $\eta_i(\bbb)=0$ and  $\eta_i$ is constant by \eqref{eq:deg-og-diff}; thus in this case $\deg(\xi_i) \dotplus \deg(\eta_i)=-\infty \leq n$ which proves the claim.
To obtain that $\deg(m^*(\xi))=\deg(\xi)$, we first show that when $\xi$ is a unital polynomial then the pull back takes the form
\begin{equation} \label{eq:m-tilde_form}
m^*(\xi)=1\tens \xi+\sum_{i}\xi_i\tens \xi'_i + \xi\tens 1,
\end{equation}
where $\xi_i,\xi_i'$ are unital, non-constant polynomials with $\deg(\xi_i) + \deg(\xi_i')\leq \deg(\xi)$. To see this, we expand $m^*(\xi)=\sum_{i} \xi_i\tens \eta_i$ according to the basis $(\xi_i)_{i}$ chosen above and, by what was just proven, this means  that $\deg(\xi_i) \plusdot \deg(\eta_i)\leq \deg(\xi)$. Then write $\eta_i=\xi_i'+r_i1$ with $\xi_i'$ unital and $r_i\in \RR$, and note that since $\xi$ and $(\xi_i)_{i>0}$ are unital and $\xi_0=1$ we have 
\begin{align*}
0 &=\xi(\bbb)=m^*(\xi)(\bbb,\bbb)=\eta_0(\bbb)=r_0;\\
\xi(g)&=m^*(\xi)(\bbb, g)=\eta_0(g) +\sum_{i>0}\xi_i(\bbb)\eta_i(g)=\eta_0(g);\\
\xi(g) &= m^*(\xi)(g,\bbb)=1\tens \eta_0(\bbb) +\sum_{i>0} \xi_i(g)\xi_i'(\bbb) +\sum_{i>0} \xi_i(g)r_i =\sum_{i>0} \xi_i(g)r_i 
\end{align*}
Thus, 
\begin{align*}
m^*(\xi) = 1\tens \eta_0 + \sum_{i>0} \xi_i\tens \xi_i' + \sum_{i>0} r_i\xi_i \tens \bbb =  1 \tens \xi + \sum_{i>0} \xi_i\tens \xi_i' + \xi \tens \bbb,
\end{align*}
and restricting the last sum to those $i$ for which $\xi_i'\neq 0$ we get the decomposition \eqref{eq:m-tilde_form}.
From \eqref{eq:m-tilde_form} we see that $m^*(\xi)(g, \bbb)=\xi(g)$ and hence that $\deg(m^*(\xi))\geq \deg(\xi)$ as desired.
\end{proof}

By  Proposition \ref{prop:mult-preserves-poly}, the multiplication map $m\colon G\times G\to G$ dualizes to a degree preserving map at the level of polynomial algebras, and hence the following definition makes sense.

\begin{definition}[co-multiplicativity] \label{def:comultiplicativity}
Let $G,H$ be  Mal'cev groups. We say that a linear map $\Psi\colon \operatorname{Pol}(G) \rightarrow \operatorname{Pol}(H)$ is \emph{co-multiplicative} if the following diagram commutes:
\begin{equation} \label{eq:pol_iso_uniqueness_diagram}
\xymatrix{ \operatorname{Pol}(G) \ar[rr]^{m^*} \ar[d]_{\Psi} & & \operatorname{Pol}(G^2) \ar[d]^{\Psi \otimes \Psi} \\ \operatorname{Pol}(H) \ar[rr]^{m^*} & & \operatorname{Pol}(H^2) }. 
\end{equation}

\end{definition}

\begin{rem}
By Lemma \ref{lem: pol-of-product-grps} and Proposition \ref{prop:mult-preserves-poly}, the multiplication map $m\colon G\times G \to G$ on a Mal'cev group dualizes to a (degree preserving) map $m^*\colon \Pol(G) \to \Pol(G) \tens \Pol(G)$ and it is now straight forward to check that $\Pol(G)$ is a commutative Hopf-algebra with comultiplication $m^*$, antipode $\inv^*$ and counit $\ev_\bbb$.  Note also, that in this terminology a strongly unital, comultiplicative algebra homomorphism between polynomial algebras is nothing but a morphism in the category of Hopf algebras. In what follows we will therefore stick to the  Hopf-algebraic language and use the term `morphism of Hopf algebras' rather than the more ad hoc terminology `strongly unital, co-multiplicative, 
algebra homomorphism'.
\end{rem}

\begin{lemma} \label{lma:pol_iso_uniqueness}
Let $G$ and $H$ be 
Mal'cev groups and let $\Psi_1,\Psi_2 \colon \operatorname{Pol}(G) \rightarrow \operatorname{Pol}(H)$ be Hopf algebra homomorphisms.
 Let $(g_{i,j})$ be a Mal'cev basis for $G$ and $(h_{k,l})$ be a Mal'cev basis for $H$ and suppose that for all $\ell=1,\dots ,\operatorname{rk}(H)$ and all $i,j$ we have
\begin{equation} \label{eq:pol_iso_uniqueness_assumption}
(\Psi_1 \zeta_{g_{i,j}})(h_{1,\ell}) = (\Psi_2 \zeta_{g_{i,j}})(h_{1,\ell}).
\end{equation}
Then $\Psi_1 = \Psi_2$.
\end{lemma}

\begin{proof}
We show that $\Psi_1(\xi) = \Psi_2(\xi)$ by induction on $d:=\deg \xi$. The case $d=0$ is clear and the case $d=1$ follows directly from the hypotheses 
using Theorem \ref{thm:pol_Malcev_description}. Let $d> 1$ be given. Suppose that $(\Psi_1 \xi)(h) = (\Psi_2 \xi)(h)$ and $(\Psi_1 \xi)(k) = (\Psi_2 \xi)(k)$ for some $h,k\in H$. Writing 
$m^*(\xi)=1\tens \xi+\sum_{i}\xi_i\tens \xi'_i + \xi\tens 1$ as in \eqref{eq:m-tilde_form}, the induction hypothesis gives:
\begin{align*}
\Psi_1(\xi)(hk) & = m^*(\Psi_1(\xi))( h,k ) \\
 & = (\Psi_1\otimes \Psi_1)(m^*(\xi))( h, k ) \\
 & = \sum_{i} \Psi_1(\xi_{i})( h ) \cdot \Psi_1(\xi_{i}')( k ) + \Psi_1(\xi)(h) + \Psi_1(\xi)(k)  \\
 & = \sum_{i} \Psi_2(\xi_{i})( h ) \cdot \Psi_2(\xi_{i}')( k ) + \Psi_2(\xi)(h) + \Psi_2(\xi)(k) \\
 & = (\Psi_2\otimes \Psi_2)(m^*(\xi))( h, k ) \\
 & = \Psi_2(\xi)(hk).
\end{align*}
Using this computation repeatedly,  the assumption \eqref{eq:pol_iso_uniqueness_assumption} implies that $(\Psi_1\xi)(h) = (\Psi_2\xi)(h)$ for all words in  $h_{1,\ell}$, and by Lemma \ref{lma:pol_determined_subgroup} it follows that $\Psi_1\xi = \Psi_2\xi$.
\end{proof}

Observe that if $\varphi \colon H\rightarrow G$ is a homomorphism then it induces a  (degree-preserving)  homomorphism $\varphi^* \colon \operatorname{Pol}(G) \rightarrow \operatorname{Pol}(H)$ of Hopf algebras. The next result gives a converse to this, in the spirit that $\operatorname{Pol}(G)$ acts as a ``total'' dual space of $G$.
As mentioned already, by using the identification of $\Pol(G)$ with the algebra of regular functions the result can also be deduced from classical results in the theory of algebraic groups \cite[Chapter 1]{Borel-GTM}.

\begin{theorem} \label{thm:hom_existence}
Let $G$ and $H$ be 
csc nilpotent Lie groups
 and suppose that $\Psi \colon \operatorname{Pol}(G) \rightarrow \operatorname{Pol}(H)$ is a  Hopf algebra homomorphism. Then there is a unique continuous group homomorphism $\varphi\colon H \rightarrow G$ such that $\Psi$ is induced by $\varphi$. Further, $\varphi$ is an isomorphism if and only if $\Psi$ is 
bijective.
\end{theorem}

\begin{proof}
 Fix Mal'cev bases $\{ g_{i,j}\}$ and  $\{ h_{i,j}\}$ for $G$ and $H$, respectively, and let $F$ be the free csc nilpotent Lie group of class $\operatorname{cl}(F) = \max \{ \operatorname{cl}(G), \operatorname{cl}(H) \}$ with $\operatorname{rk}(H)$ generators $f_{1,1},\dots ,f_{1,\operatorname{rk}(H)}$. Then there 
 are unique Lie group homomorphisms $\varphi_H\colon F\to H$ and $\varphi_G\colon F\to G$ with closed images,  defined on the generators by
 \begin{align*}
 \varphi_H(f_{1,\ell}) &:=h_{1,\ell}\\
 \varphi_G(f_{1,\ell}) &:= \prod_{(i,j)\in \Bultiindex{B}{\multirank(G)}} g_{i,j}^{(\Psi \zeta_{g_{i,j}})(h_{1,\ell})}
 \end{align*}
 Moreover, since $\varphi_H$ arises as the predual of a surjective map at the level of Lie algebras, $\varphi_H$ is also surjective and
hence  induces an isomorphism $F/\ker(\varphi_H) \simeq H$. A direct computation shows that $(\varphi_H^* \circ \Psi)(\zeta_{g_{ij}})(f_{1l}) =\varphi_G^*(\zeta_{g_{ij}}) (f_{1l}) $ and thus, by Lemma \ref{lma:pol_iso_uniqueness}, we get $\varphi_H^*\circ \Psi=\varphi_G^*$. Let $f\in F$ be in $\ker(\varphi_H)$. Then for every $\zeta\in \Pol^0(G)$ we have, since $\Psi$ is assumed strongly unital, that
\[
0=\Psi(\zeta) (\bbb)=\Psi(\zeta) (\varphi_H(f))= (\varphi_H^*\circ \Psi)(\zeta)(f)=\varphi_G^*(\zeta)(f)=\zeta(\varphi_G(f)),
\]
and since $\Pol^{0}(G)$ separates points in $G$ (Theorem \ref{thm:pol_Malcev_description}) we conclude that $\varphi_G(f)=\bbb$. Thus $\varphi_G$ induces a map $\bar{\varphi}_G\colon F/\ker(\varphi_H)\to G$ and we therefore obtain a homomorphism
\[
\varphi\colon H \simeq F/\ker(\varphi_H)\overset{\bar{\varphi}_G}{\To} G
\]
Note that, since $\varphi_G$ has closed image the same is true for $\varphi$. By construction we have $\varphi(h_{1l})=\varphi_G(f_{1l})$ and a direct computations now shows that $\varphi^*(\zeta_{g_{ij}})(h_{1l})=\Psi(\zeta_{g_{ij}})(h_{1l})$ for all $i,j$ and $l$ and by Lemma \ref{lma:pol_iso_uniqueness} we conclude that $\varphi^*=\Psi$.  This also proves the uniqueness of $\varphi$, because if $\psi$ were another homomorphism predual to $\Psi$ then for every $h\in H$ and every $\xi\in \Pol(G)$ we have
\[
\xi(\psi(h))=\psi^*(\xi)(h)=\Psi(\xi)(h)=\varphi^*(\xi)(h)=\xi(\varphi(h)),
\]
and since $\Pol(G)$ separates points in $G$ we conclude that $\varphi(h)=\psi(h)$. If $\Psi$ is moreover assumed bijective, then $\Psi^{-1}$ is also a Hopf algebra homomorphism and is therefore induced by a unique group homomorphism $\psi\colon H\to G$. Again by the uniqueness of the homomorphism, it follows that $\psi\circ \varphi=\id_G$ and $\psi\circ\varphi=\id_H$.
\end{proof}

\begin{rem}
If $G$ and $H$ are Mal'cev groups then by Theorem \ref{thm:pol_Malcev_description}, the restriction map $\Pol(G\tens \RR) \to \Pol(G)$ is a (degree preserving) Hopf algebra isomorphism.  Thus, if $\Psi\colon \Pol(G) \to \Pol(H)$ is a Hopf algebra homomorphism then, by Theorem \ref{thm:hom_existence}, it is induced by a group homomorphism $\psi\colon H\otimes \RR \to G\otimes \RR$ at the level of Mal'cev completions which is an isomorphism exactly when $\Psi$ is bijective.
\end{rem}


\end{document}